\documentclass[graybox]{svmult}

\usepackage{type1cm}        
\usepackage{makeidx}         
\usepackage{graphicx}        
\usepackage{multicol}        
\usepackage[bottom]{footmisc}
\usepackage{hyperref}

\usepackage{newtxtext}       %
\usepackage[varvw]{newtxmath}       

\makeindex             

\usepackage{bookmark}
\makeatletter
\def\toclevel@title{-1}
\def\toclevel@author{0}
\makeatother

\begin{document}

\title*{Liouville theorems and
Evans-Krylov estimates}

\author{Yifan Chen\orcidID{0000-0001-8816-3776},\\
Hans-Joachim Hein\orcidID{0000-0002-3719-9549},\\
Ryan Mc Gowan}


\institute{Yifan Chen \at Dipartimento di Matematica, Universit\`a di Roma Tor Vergata, Via della Ricerca Scientifica 1, I-00133 Roma, Italy \\ e-mail: \href{mailto:chen@mat.uniroma2.it}{chen@mat.uniroma2.it}
\and Hans-Joachim Hein \at Mathematisches Institut, Universität Münster, 48149 M\"unster, Germany\\ e-mail: \href{mailto:hhein@uni-muenster.de}{hhein@uni-muenster.de}
\and Ryan Mc Gowan \at Department of Mathematics, Rutgers University, Piscataway, NJ 08854, USA \\ e-mail: \href{mailto:rmm396@math.rutgers.edu}{rmm396@math.rutgers.edu}}

\maketitle

\abstract{A classical idea in analysis going back at least to work of Simon (1997) is that Liouville theorems for solutions to elliptic or parabolic PDEs are equivalent to Schauder type regularity estimates. The goal of this course is to describe some recent developments of this idea concerning the regularity of the complex Monge-Amp\`ere equation with respect to singular reference metrics. We will start with a quick look at the classical $C^2$ and $C^3$ estimates of Calabi-Aubin-Yau and then present a new proof of the Evans-Krylov $C^{2,\alpha}$ estimate on a Euclidean ball. Based on this we will consider the case of singular backgrounds such as cylinders and cones, discussing some recent work by Hein, Tosatti, Lee and Klemmensen. Our discussion is far from complete and knowledge of Kähler geometry including the Aubin-Yau theorems is assumed. The course includes five exercises with solutions and a problem list.}

\bigskip

\noindent Lecture notes for a mini-course by Hans-Joachim Hein at the Erd\H{o}s Center--Alfr\'ed R\'enyi Institute of Mathematics, Budapest, August 11-15, 2025.

\bigskip

\begin{enumerate}
    \item Classical $C^2$ and $C^3$ estimates and Yau's Liouville Theorem
    \item From Liouville to Evans-Krylov, with applications to cylinders
    \item Calabi-Yau cones and Klemmensen's thesis
    \item Open problems
    \item Solutions to the exercises
\end{enumerate}

\section{Classical \texorpdfstring{$C^2$}{C^2} and \texorpdfstring{$C^3$}{C^3} estimates and Yau's Liouville Theorem}

\subsection{The complex Monge-Ampère equation on a ball in \texorpdfstring{$\mathbb{C}^n$}{C^n}}

We want to consider the complex Monge-Amp\`ere equation
\begin{equation}\label{CMA}
    \det\left(\frac{\partial^2 u}{\partial z_j \partial \bar z_k}\right)=1
\end{equation}
for some smooth function $u:B\to \mathbb{R}$, where $B=B_1(0)\subset \mathbb{C}^n$ denotes the open unit ball in $\mathbb{C}^n$. The matrix in parentheses is automatically Hermitian. We also require it to be nonnegative definite, which means by definition that $u$ is a \emph{plurisubharmonic} (psh) function. This requirement makes \eqref{CMA} an elliptic PDE and is equivalent to the statement that $u$ is the K\"ahler potential of a K\"ahler metric $\omega_u$ on $B$,
\begin{equation*}
    \omega_u = \sqrt{-1} \partial\bar \partial  u = \sqrt{-1}\sum_{j,k=1}^n \frac{\partial^2 u}{\partial z_j \partial \bar z_k} dz_j\wedge d\bar z_k.
\end{equation*}

\noindent \textbf{Question.} Why do we want to consider \eqref{CMA}?
\begin{enumerate}
    \item It allows us to see when a K\"ahler metric is Ricci-flat. Indeed, with $\omega_u$ as above,
    \begin{equation*}
        \mathrm{Ric}_{\omega_u} = -\sqrt{-1}\partial\bar \partial \log(\omega_u^n)=0.
    \end{equation*}
    \item Given $\tilde u :B_1(0)\subset \mathbb{R}^n \to \mathbb{R}$ convex solving the real Monge-Amp\`ere equation
    \begin{equation*}
        \det(D^2\tilde u ) =1,
    \end{equation*}
    then $u(z) := \tilde u(\mathrm{Re}(z))$ is psh and it solves \eqref{CMA}. The associated Ricci-flat K\"ahler metrics $\omega_u$ are exactly the so-called \emph{semi-flat Calabi-Yau metrics}, which play a key role in the SYZ conjecture \cite{YangLi}. On the other hand, the real Monge-Amp\`ere equation has been studied extensively from a PDE point of view \cite{Figalli}.
\end{enumerate}

\noindent Fortunately, then, there exist many solutions to \eqref{CMA}: 

\begin{theorem}[Regularity of the Dirichlet Problem]\label{thm:dir}
    Fix any $\varphi \in C^0(\partial B)$.
    \begin{enumerate}
    \item (Bedford-Taylor \cite{Bedford_Taylor_1976}) There exists a unique $u\in C^0(\overline B)$ with $u|_{\partial B}=\varphi$ such that $u|_{B}$ is psh in the weak sense and solves \eqref{CMA} in the weak sense. 
    \item (Caffarelli-Kohn-Nirenberg-Spruck \cite{CKNS}) If $\varphi \in C^\infty(\partial B)$, then $u\in C^\infty(\overline B)$.  
\end{enumerate}
\end{theorem}

Please refer to the solution of Exercise 1 in Section \ref{sec:solutions} for a discussion of weakly psh functions and weak solutions of \eqref{CMA}. Here, we would first like to emphasize the following point: whereas results such as Theorem \ref{thm:dir} are familiar from the standard Dirichlet problem for linear elliptic PDEs such as the Laplace equation, such linear PDEs enjoy the following additional interior regularity property:
\begin{equation*}
    \varphi \in C^0(\partial B)\implies u\in C^\infty_{\rm loc}(B).
\end{equation*}
This is fundamentally false for Equation \eqref{CMA}:

\begin{example}\label{blocki} Blocki \cite{blocki1999regularity} pointed out the example $u(z_1,z_2) = (1+|z_1|^2)|z_2|$ on $\overline B \subset \mathbb{C}^2$. Your {\bf first exercise} is to check that $u$ solves \eqref{CMA} weakly and to check the following geometric interpretation of this example: $u$ is the pullback of the potential of a flat Kähler metric on $\mathbb{C}^2/\{\pm \textrm{Id}\}$ to the crepant blowup of this singularity.
\end{example}

\begin{example} From Mooney \cite{Mooney}, convex weak solutions $\tilde{u}: B_1(0) \subset \mathbb{R}^n \to \mathbb{R}$ to the real Monge-Amp\`ere equation $\det(D^2\tilde{u}) = 1$ can be singular at worst along a subset of Hausdorff codimension $1$, but for $n \geq 3$ this is actually sharp. Thus, for $n \geq 3$ we obtain singular weak solutions to \eqref{CMA} on $B \subset \mathbb{C}^n$ with codimension $1$ singular set via $u(z) = \tilde{u}({\rm Re}(z))$. Note that, here, the real codimension of the singular set is $1$, while in Example \ref{blocki} the complex codimension of the singular set is $1$.
\end{example}

In contrast to the real case, it seems to be unknown whether a weak solution to Equation \eqref{CMA} as in Theorem \ref{thm:dir}(1) has any smooth points at all.

\medskip

\noindent \textbf{Question.} How did \cite{CKNS} prove that $\varphi\in C^\infty(\partial B)$ implies $u\in C^\infty(\overline B)$?

\medskip

\noindent Their approach was to use the {continuity method} to construct a smooth solution $u$, then invoke the uniqueness of solutions. Specifically, extend $\varphi$ to a psh function $\overline \varphi$ on $\overline B$, e.g. by taking any extension $\hat \varphi\in C^\infty(\overline B)$ and setting $\overline \varphi(z) = \hat \varphi(z) +\lambda (|z|^2-1)$ for $\lambda\gg1$.
Then consider the family of PDEs
\begin{equation*}
\det\left(\dfrac{\partial^2u_t}{\partial z_j\partial \bar z_k}\right) = (1-t)\det\left(\dfrac{\partial^2\overline\varphi}{\partial z_j\partial \bar z_k}\right)+t\;\text{on }B,\quad u_t\;\text{psh on }B, \quad u_t = \varphi\;\text{on } \partial B 
\end{equation*}
for $t \in [0,1]$. This is trivially solvable for $t = 0$. Openness of the set of $t$-values for which a solution exists is an easy consequence of the Dirichlet problem for linear elliptic PDEs. So, as usual, the key to being able to solve all the way up to $t = 1$ is to establish \textit{a priori} estimates of the solutions $u_t$. Pretend that the right-hand side is identically equal to $1$. The following is a sketch of some elements of the \emph{a priori} estimates. Your \textbf{second exercise} is to check these for hidden traps ((i)--(v)).

\subsection{Classical \texorpdfstring{$C^2$}{C^2} estimates}

For notational simplicity, we write 
\begin{equation*}
    \frac{\partial u}{\partial z_j} = u_j, \qquad\frac{\partial u}{\partial\bar z_k} = u_{\bar k}.
\end{equation*}
We write $ u^{\bar jk}$ to denote the $(j,k)$ entry of the inverse matrix of
\begin{equation*}
    (u_{j\bar k}) = \left(\dfrac{\partial^2u}{\partial z_j\partial \bar z_k}\right).
\end{equation*}
Furthermore, we sum over repeated indices as usual.

We now want to understand the classical $C^2$ estimates in more detail. Note that we can write \eqref{CMA} in the following equivalent form:
\begin{equation*}
    \log\det\left(\dfrac{\partial^2u}{\partial z_j\partial \bar z_k}\right)=0.
\end{equation*}
If we take the $\ell$ derivative in the formula above, we get
\begin{equation*}
    u^{\overline jk}u_{k\overline j\ell} = 0.
\end{equation*}
Thus, taking the $\bar m$ derivative, we get 
\begin{align*}
    0
    &= -u^{\bar j p}u_{p\bar q\bar m}u^{\bar q k}u_{k\bar j\ell}+u^{\bar jk}u_{k\bar j \ell \bar m}\\
    &= -u^{\bar j p}u_{p\bar q\bar m}u^{\bar q k}u_{k\bar j\ell}+L(u_{\ell\bar m}).    
\end{align*}
Here and below, $L$ denotes the linearized Monge-Amp\`ere operator
\begin{equation*}
    L(f) = u^{\bar jk}f_{k\bar j}.
\end{equation*}
It is not difficult to see that this operator is a real, second-order elliptic operator (in fact, it is the Laplace-Beltrami operator $\Delta_{\omega_u}$ of the K\"ahler metric $\omega_u$), and hence it satisfies the maximum principle. Rearranging, we see that
\begin{equation}\label{eq:C2:est}
    L(u_{\ell\bar m}) = u^{\bar j p}u_{p\bar q \bar m}u^{\bar q k}u_{k\bar j\ell }.
\end{equation} 
We can rotate coordinates to make $ u_{j\bar k} = \lambda_j \delta_{j\bar k} $ for some $\lambda_j\in \mathbb{R}^+$ (i), thus giving
\begin{equation*}
    L(u_{\ell \bar m}) = \sum_{j,k=1}^n \frac1{\lambda_j} \frac1{\lambda_k} u_{j\bar k\bar m}u_{\bar j k\ell}.
\end{equation*}
If $\ell = m$, then this is nonnegative because $u_{j \bar k \bar \ell} u_{\bar j k \ell} = |u_{j \bar k \bar \ell} |^2$. Thus, summing over $\ell=m$, we have $L(\Delta u)\geq 0$, where $\Delta = \Delta_{\omega_{\rm Euc}}$ is the standard Euclidean Laplacian. Via the maximum principle,  $\Delta u$ attains its maximum on the boundary $\partial B$. Since $\varphi \in C^\infty(\partial B)$, the tangential contributions to $\Delta u$ on $\partial B$ are under control.

Unlike the case of linear elliptic PDEs, the normal contributions to $\Delta u$ on $\partial B$ cannot be computed directly from the PDE, but they can be controlled using barrier functions and computations similar to the above. This is, in fact, one of the main innovations of \cite{CKNS} over Aubin \cite{Aubin} and Yau \cite{Yau}, but we shall not discuss this.

The upshot is that we have an \emph{a priori} estimate 
\begin{equation*}
    \sup_{\overline B}\ \max \left\{\lambda_1, \frac{1}{\lambda_1}, \ldots, \lambda_n, \frac{1}{\lambda_n}\right\} \leq C(\varphi),
\end{equation*}
where $\lambda_j$ are the eigenvalues of $(u_{j\bar{k}})$ and $C(\varphi)$ is independent of $t$.

\medskip

\begin{remark}\label{rem:C^2sing}
    Despite some nice progress in special cases, it is still a fundamental open problem to make these arguments work in a singular reference geometry. Precisely, split $u = u_{\textrm{Ref}} + u_{\textrm{Unk}}$, where $\omega_{\textrm{Ref}} = \omega_{u_{\textrm{Ref}}}$ is a chosen reference K\"ahler metric and $u_{\textrm{Unk}}$ is the unknown part of the solution. Now aiming to control $\Delta_{\omega_{\textrm{Ref}}} u = n + \Delta_{\omega_{\mathrm{Ref}}} u_{\mathrm{Unk}}$ instead of the Euclidean Laplacian $\Delta u$, work in normal coordinates with respect to $\omega_{\textrm{Ref}}$. Then $L(\Delta u)$ differs from $L(\Delta_{\omega_{\textrm{Ref}}} u)$ by 
\begin{equation*}
    L(u_{\mathrm{Ref}}^{\bar{m}\ell})u_{\ell\bar{m}} = u^{\bar{j}k}(u_{\mathrm{Ref}}^{\bar{m}\ell})_{k\bar{j}} u_{\ell\bar{m}} = -\sum_{j,\ell=1}^n \frac{\lambda_\ell}{\lambda_j}u_{\mathrm{Ref},\ell\bar{\ell}j\bar{j}}.
\end{equation*}
    The coefficients appearing in this expression are the sectional curvatures of $\omega_{\mathrm{Ref}}$, which are typically unbounded on a singular space. For two situations where these terms can be dealt with, see \cite{GuenanciaPaun, Tosatti2010} and references therein. For 
    singularities such as the cones in Section \ref{sec:cones}, no workaround of this kind is known.
\end{remark}

\subsection{Classical \texorpdfstring{$C^3$}{C^3} estimates} 

Deviating slightly from \cite{CKNS}, we are now going to follow the classical Calabi-Aubin-Yau $C^3$ computation \cite{Calabi, Aubin, Yau} as rewritten by Phong-\v Se\v sum-Sturm \cite{PhongSesumSturm}. The basic idea is to interpret the $C^2$ computation above as a kind of Bochner formula. When working with Bochner type formulas, to get an estimate of order $k+1$ one typically applies $L$ to the positive term $f$ in the estimate of order $k$. This is almost what we do here, except that we contract with respect to $\omega_u$ rather than $\omega_{\rm Euc}$:
\begin{equation*}
    f = u^{\bar m\ell}\, u^{\bar j p}u_{p\bar q\bar m}u^{\bar q k}u_{k\bar j\ell}.
\end{equation*}
Before computing $L(f)$, we rewrite $f$ a bit more geometrically via the Levi-Civita connection of $\omega_u$ by introducing the Christoffel symbols
\begin{equation*}
    \frac{\partial g_{j\bar k}}{\partial z_\ell} = \Gamma_{j\ell}^r g_{r\bar k}.
\end{equation*}
Hence, we have, after contraction,
\begin{align*}
    f &= u^{\bar m\ell}u^{\bar jp}\overline{\Gamma_{qm}^ru_{r\bar p}} u^{\bar qk}\Gamma^s_{k\ell}u_{s\bar j}\\
    &= u^{\bar m\ell}\overline{\Gamma^j_{qm}}u^{\bar qk}\Gamma^s_{k\ell}u_{s\bar j}\\
    &= |T|^2_{\omega_u},    
\end{align*}
where the tensor $T$ (ii) is given by 
\begin{equation*}
    T = (\nabla^{\rm LC}_{\omega_u} -\nabla^{\rm LC}_{\omega_{\mathrm{Euc}}})|_{T^{1,0}(\mathbb{C}^n)}.
\end{equation*} 
Hence, applying the operator $L$, we find (iii)
\begin{equation*}
    L(f) = \Delta_{\omega_u}|T|^2_{\omega_u} = |\nabla_{\omega_u} T|_{\omega_u}^2 +\langle\Delta_{\omega_u}T,T\rangle_{\omega_u}.
\end{equation*}
If we use holomorphic normal coordinates for $\omega_u$ at a point, we get
\begin{equation*}
    (\Delta_{\omega_u}T)^\ell_{jk} = \Delta_{\omega_u}(T^\ell_{jk}) - T\left(\Delta_{\omega_u}\left(\dfrac{\partial}{\partial z_j}\right), \frac{\partial}{\partial z_k}, dz^\ell\right)-\cdots .
\end{equation*}
Because $\omega_u$ is Ricci-flat, the second and subsequent terms vanish by the Bochner formulas for holomorphic vector fields and $1$-forms (iv). We are then left with 
\begin{align*}
    (\Delta_{\omega_u}T)^\ell_{jk} &= (u^{\bar m\ell}u_{j\bar m k}- (\Gamma_{\rm Euc})^\ell_{jk})_{a\bar a}\\
    &= u_{j\bar \ell k a\bar a} - \frac{\partial }{\partial z^a}({\rm Rm}_{\mathrm{Euc}})^\ell_{j\bar ak}.
\end{align*}
Clearly ${\rm Rm}_{\mathrm{Euc}} = 0$. Now, by the $C^2$ estimates \eqref{eq:C2:est}, we see (v)
\begin{align*}
    (\Delta_{\omega_u} T)^\ell_{jk} &= (u^{\bar ab}u_{k\bar \ell b\bar a})_j\\
    &= (u^{\bar rp}u_{p\bar q \bar \ell}u^{\bar qs}u_{s\bar r k})_j = 0.
\end{align*}
Combining everything together, we see that
\begin{equation}\label{eq:C3:est}
L(f) = \Delta_{\omega_u} |T|_{\omega_u}^2 = |\nabla_{\omega_u} T|_{\omega_u}^2,
\end{equation}
i.e., we again have the subharmonicity expected of a Bochner type quantity.

\begin{remark}
    Again replacing $\omega_{\mathrm{Euc}}$ by a possibly singular reference metric $\omega_{\mathrm{Ref}}$ as in Remark \ref{rem:C^2sing}, we see that, now, the first derivative of the curvature of $\omega_{\mathrm{Ref}}$ enters the estimate, as opposed to merely the curvature itself. On a singular space this clearly makes matters worse. On the other hand, the $C^3$ estimate is slightly more linear in nature than the $C^2$ estimate. The moral of the rest of this course will be to show that, even in some quite singular reference geometries, this increased linearity wins.
\end{remark}

\subsection{Yau's Liouville Theorem}\label{sec:YauLTOld}

As a consequence of the $C^2$ and $C^3$ estimate computations for the complex Monge-Amp\`ere equation, we obtain the following result, due to Riebesehl-Schulz \cite{RS}.

\begin{theorem}[Yau's Liouville Theorem\label{cor:YauLT}]
Let $\omega$ be a K\"ahler metric on $\mathbb{C}^n$. Then 
\begin{equation*} 
\left\{\begin{array}{cc}
     \mathrm{(I)} & \omega^n =\omega_{\mathrm{Euc}}^n\\
     \mathrm{(II)} &\dfrac1C\omega_{\mathrm{Euc}}\leq \omega\leq C\omega_{\mathrm{Euc}}
\end{array}\right\} \implies \exists A\in \mathrm{SL}(n, \mathbb{C}): \omega=A^*\omega_{\mathrm{Euc}}.
\end{equation*}
\end{theorem}

\begin{example}[LeBrun \cite{lebrun}]
If we drop Condition (II), the result need not hold. One can show that, on $\mathbb{C}^2$, if we consider the Taub-NUT metric $\omega_u =\sqrt{-1}\partial\bar \partial u$,
\begin{equation*} 
u(z_1,z_2) = (a^2+a^4)+(b^2+b^4), \;\,\text{where}\;\; \left\{ 
\begin{array}{c}|z_1|=ae^{a^2-b^2}\\ |z_2|=be^{b^2-a^2}\end{array}\right\}, 
\end{equation*}
then (I) holds but $\omega_u$ is not flat. However, it is complete, of rapid curvature decay at infinity, and is thus a \emph{gravitational instanton}. Your {\bf third exercise} is to check this.
\end{example}

\begin{example}[Calabi, personal communication, 2013]
Calabi observed a very different kind of counterexample to Corollary \ref{cor:YauLT} without condition (II). There exists a domain $V \subset \mathbb{C}^2$, $V \neq \mathbb{C}^2$, together with a biholomorphism $F: \mathbb{C}^2 \to V$ such that $\det DF = 1$; see \cite{RoRu} for a modern discussion. The counterexample is then $\omega = F^*(\omega_{\mathrm{Euc}}|_V)$, and is flat but incomplete. Domains of this kind are called Fatou-Bieberbach domains and the known examples have fractal boundaries. The fact that $\det DF = 1$ is possible for a Fatou-Bieberbach map $F$ was known to Bieberbach \cite{Bieber}.  
\end{example}

\begin{remark}
    For the real Monge-Amp\`ere equation, examples such as the above do not occur: an entire convex solution $\tilde{u}: \mathbb{R}^n \to \mathbb{R}$ to $\det(D^2\tilde{u})=1$ is automatically a quadratic polynomial. This was first proved by Jörgens \cite{Jorgens} for $n = 2$, by Calabi \cite{Calabi} for $n = 3,4,5$ and by Pogorelov \cite{Pogo} for $n \geq 6$. Amusingly, for $n = 2$ the statement is equivalent to Bernstein's theorem: a function $f: \mathbb{R}^2 \to \mathbb{R}$ such that ${\rm graph}(f)$ is a minimal surface in $\mathbb{R}^3$ is affine linear. While Bernstein's theorem holds for minimal graphs over $\mathbb{R}^n$ for $3 \leq n \leq 7$ as well, it fails for $n \geq 8$.
\end{remark}

\begin{proof}[Proof of Theorem \ref{cor:YauLT}]
This is again a standard argument in the context of Bochner type formulas. Denote $\omega_{\mathrm{Euc}}$ by $\tilde\omega$ for notational simplicity. Choose a cutoff function $\chi=\chi(|z|)\in C^\infty(\mathbb{R})$ as in Figure \ref{cutoff} such that $\chi(|z|)\equiv 1$ for $|z|\leq R$, $\chi(|z|)\equiv 0$ outside a neighborhood of $|z|\leq 2R$, and $\chi(|z|)\geq 0$ for $R<|z|<2R$, as shown below.
    \begin{figure}[b]
        \sidecaption
        \includegraphics[scale=.5]{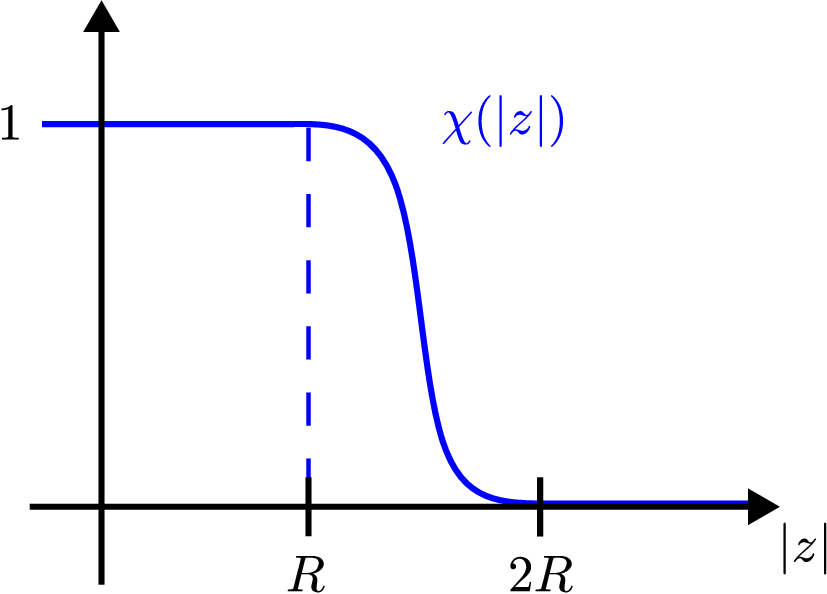}
        \caption{The cutoff function $\chi(|z|)$.}
        \label{cutoff}
    \end{figure}
Treat $C$ as a generic constant. Choose some $\hat C\gg C$ (which shall become clear later in the proof), and  consider applying the operator $\Delta_{\omega}$ as follows:
    \begin{align*}
    \Delta_{\omega}\Bigg(\chi^2 |T|^2_\omega + \dfrac{\hat C}{R^2} \mathrm{Tr}_{\tilde\omega}(\omega)\Bigg)
    &=
    \left(\Delta_\omega  \chi^2\right) |T|^2_\omega
    +2\left\langle\nabla_\omega  \chi^2, \nabla_\omega |T|^2_\omega\right\rangle_\omega\\
    &+\chi^2 \Delta_\omega |T|^2_\omega+\frac{\hat C}{R^2}\Bigg( \Delta_{\omega}\mathrm{Tr}_{\tilde\omega}(\omega)\Bigg).    
    \end{align*}
Our goal is to show that this quantity is nonnegative. To this end, we estimate each of the four terms on the right-hand side, respectively, as follows:
\begin{enumerate}
    \item By Condition (II) we have $\Delta_\omega \chi^2\geq -C/R^2$.
    \item Applying Young's inequality and Condition (II) we have 
    \begin{equation*} 
    2\left\langle\nabla_\omega  \chi^2, \nabla_\omega  |T|^2_\omega\right\rangle_\omega \geq -\frac C{R^2} |T|^2_\omega -\frac{\chi^2}{100}|\nabla_\omega  |T|_\omega|^2_\omega.
    \end{equation*}
    \item By the $C^3$ estimate \eqref{eq:C3:est} we have $\Delta_\omega  |T|^2_\omega=|\nabla_\omega  T|^2_\omega$.
    \item By the $C^2$ estimate \eqref{eq:C2:est} and Condition (II) we have 
    $\Delta_\omega \mathrm{Tr}_{\tilde\omega}(\omega)\geq |T|_\omega^2/C$.
\end{enumerate}

\noindent Putting everything together, we see 
\begin{equation*}
\Delta_{\omega}\left( \chi^2|T|^2_\omega + \frac{\hat C}{R^2}\mathrm{Tr}_{\tilde\omega}(\omega)\right) \geq 
\left(\frac{\hat C}{R^2 C}-\frac{C}{R^2}\right) |T|_\omega^2 +\frac{99}{100}\chi^2 |\nabla_\omega  T|^2_\omega.
\end{equation*}
Note that if $\hat C\gg C$ then the above quantity is non-negative, as required. Thus, we have, by applying the maximum principle, that:
\begin{equation*} 
\sup_{B_R(0)} |T|^2_\omega \leq \frac{\hat C}{R^2}\sup_{\partial B_{2R}(0)}\mathrm{Tr}_{\tilde\omega}(\omega)\leq \frac{\hat{C}^2}{R^2}. 
\end{equation*}
Taking $R\to \infty$ gives $T = \nabla_{\mathrm{Euc}}\omega=0$, which completes the proof. 
\end{proof}

\section{From Liouville to Evans-Krylov, with applications to cylinders}

We now reverse the logic of the previous section. We begin by giving a self-contained proof of the Liouville Theorem without using the $C^3$ estimate computations. From this, a $C^{2,\alpha}$ estimate can be proved by contradiction. This constitutes a new proof of the Evans-Krylov estimate, first observed in \cite{HeinTosatti2020}, which can be generalized to certain singular spaces to which the classical proofs do not seem to apply.

\subsection{A geometric proof of Yau's Liouville Theorem}\label{sec:YauLT}

This is due to Li-Li-Zhang \cite{lilizhang} after some similar but more complicated arguments in \cite{Hein}, which prove less. In particular, the heat kernel estimates developed in \cite{Hein} can be packaged much more efficiently in the form of Peter Li's Theorem \ref{lithm} below.

Let $\xi\in\mathbb{C}^n$, $\xi \neq 0$, and view this as a constant $(1,0)$-vector field on $\mathbb{C}^n$. Hence, 
by the Bochner formula for vector fields, we have
\begin{equation*} 
\Delta_\omega|\xi|^2_\omega =-\langle (\overline{\partial}^*_\omega \overline{\partial} + {\rm Ric}_\omega)\xi, \xi\rangle_\omega + |\nabla_\omega\xi|_\omega^2 \geq 0. 
\end{equation*}
This is in fact an equivalent way of writing the $C^2$ estimate calculation \eqref{eq:C2:est}, so we are not eliminating this calculation from the proof of the Liouville theorem. However, assuming this, the arguments below will truly bypass the $C^3$ estimate calculation.
Note that $|\xi|_\omega^2\in L^\infty(\mathbb{C}^n)$ by Condition (II). We shall require the following result:

\begin{theorem}[Li \cite{li}]\label{lithm}
For all complete, non-compact Riemannian manifolds $(M,g)$ with $\mathrm{Ric}_g\geq 0$ and for all smooth functions $f\in L^\infty(M)$ with $\Delta_g f\geq 0$, we have 
\begin{equation*} 
\sup_{M} f =\lim_{R\to \infty}\frac{1}{\mathscr{A}_g(\partial B_R^g(p))}\int_{\partial B_R^g(p)}f\, d\mathscr{A}_g 
\end{equation*}
for any fixed $p\in M$.
\end{theorem}

For us, this theorem is applicable because $(\mathbb{C}^n,\omega)$ is complete by Condition (II) and Hopf-Rinow. The left-hand side defines a norm $\|\cdot\|$ on $\mathbb{C}^n$:
\begin{equation*}
\|\xi\|^2=\sup_{\mathbb{C}^n}|\xi|^2_\omega.
\end{equation*}  
The triangle inequality follows from Cauchy-Schwarz.
Then, by Theorem \ref{lithm},
\begin{equation*} 
\|\xi\|^2=\lim_{R\to \infty}\frac{1}{\mathscr{A}_\omega(\partial B_R^\omega(0))}\int_{\partial B_R^\omega(0)}|\xi|^2_\omega\, d\mathscr{A}_\omega . 
\end{equation*}
In particular, $\|\cdot\|$ satisfies the parallelogram law, hence
$\|\xi\|^2=\omega_\infty(\xi, \overline{\xi})$ for some constant K\"ahler form $\omega_\infty$ on $\mathbb{C}^n$. Now, clearly, $\omega\leq \omega_\infty$ on $\mathbb{C}^n$, so it is enough to show that $\omega^n=\omega_\infty^n$. This will follow if we can show that for any $\delta>0$ there exists a $q\in \mathbb{C}^n$ such that for all $\xi\in S :=\{e_j\pm e_k, e_j\pm \sqrt{-1}e_k: j,k=1,2,\dots, n\}$,
\begin{equation*} 
|\xi|^2_{\omega(q)}> (1-\delta)|\xi|^2_{\omega_\infty}.
\end{equation*}
Indeed, by polarization, this estimate will imply that the two quadratic forms are $O(\delta)$ close at $q$, so by continuity of the determinant the volume forms are $O(\delta)$ close at $q$; but both of these volume forms are constant, and hence are then $O(\delta)$ close everywhere. It remains to prove the above claim (which would be trivial if $q$ was allowed to depend on $\xi$ because $|\xi|_{\omega_\infty}^2 = \|\xi\|^2$ is the sup over $\mathbb{C}^n$ of $|\xi|_\omega^2$).

We can prove the claim as follows. For any $R > 0$ and $\xi \in S$ write 
$$Q_R := \partial B_R^\omega(0), \quad Q_{R,\xi} := \left\{q\in Q_R: |\xi|_{\omega(q)}^2>(1-\delta)\|\xi\|^2\right\}.$$
Then let $\delta>0$. By Theorem \ref{lithm} we have for all $R\gg1$ and for all $\xi\in S$ that
\begin{equation*}
\frac{1}{\mathscr{A}_\omega(Q_R)}\int_{Q_R}|\xi|_\omega^2\, d\mathscr{A}_\omega
> \left(1-\frac\delta{|S|}\right)\|\xi\|^2.
\end{equation*} 
It follows from this that $Q_{R,\xi}$ must occupy more than $1 - \frac{1}{|S|}$ times the total area of $Q_R$ because otherwise
\begin{align*}\int_{Q_R} |\xi|_\omega^2 \,d\mathscr{A}_\omega &\leq \|\xi\|^2 \mathscr{A}_\omega(Q_{R,\xi}) + (1-\delta)\|\xi\|^2 \mathscr{A}_\omega(Q_R \setminus Q_{R,\xi})\\
&= \delta \|\xi\|^2 \mathscr{A}_\omega(Q_{R,\xi}) + (1-\delta)\|\xi\|^2\mathscr{A}_\omega(Q_R) \\
&\leq \left(\delta\left(1 - \frac{1}{|S|}\right) + (1-\delta)\right)\|\xi\|^2 \mathscr{A}_\omega(Q_R),\end{align*}
which contradicts the previous inequality. But, if $\mathscr{A}_\omega(Q_{R,\xi}) > (1 - \frac{1}{|S|})\mathscr{A}_\omega(Q_R)$ for all $\xi \in S$, then an easy inclusion-exclusion argument shows that $\bigcap_{\xi \in S} Q_{R,\xi}$ has positive area, and hence must be nonempty.

\subsection{From Liouville to the Evans-Krylov \texorpdfstring{$C^{2,\alpha}$}{C^{2,\alpha}} estimate}\label{sec:fromLTtoEK}

We now use Yau's Liouville Theorem, which we have re-proved from first principles above, to deduce the Evans-Krylov $C^{2,\alpha}$ estimate for Equation \eqref{CMA}, which is a slight weakening of the classical $C^3$ estimate, by blowup and contradiction.

\begin{theorem}[Evans-Krylov $C^{2,\alpha}$ estimate]\label{thm:EKclass}
For all $n\in \mathbb{N}$ and $\alpha\in(0,1)$, and for all $C<\infty$, there exists a $C' = C'(n, \alpha, C)<\infty$ such that the following holds. If $\omega$ is a K\"ahler metric on $B_3(0)\subset \mathbb{C}^n$ obeying
\begin{equation*} \frac1C\omega_{\mathrm{Euc}}\leq \omega \leq C\omega_{\mathrm{Euc}}, \quad \omega^n =\omega^n_{\mathrm{Euc}},
\end{equation*}
then we have
\begin{equation*} 
|\omega(x)-\omega(y)|_{\omega_{\rm Euc}}\leq C'|x-y|^\alpha, \quad \forall x,y\in B_1(0). 
\end{equation*}
\end{theorem}

\begin{figure}[b]
    \sidecaption
    \includegraphics[scale=.475]{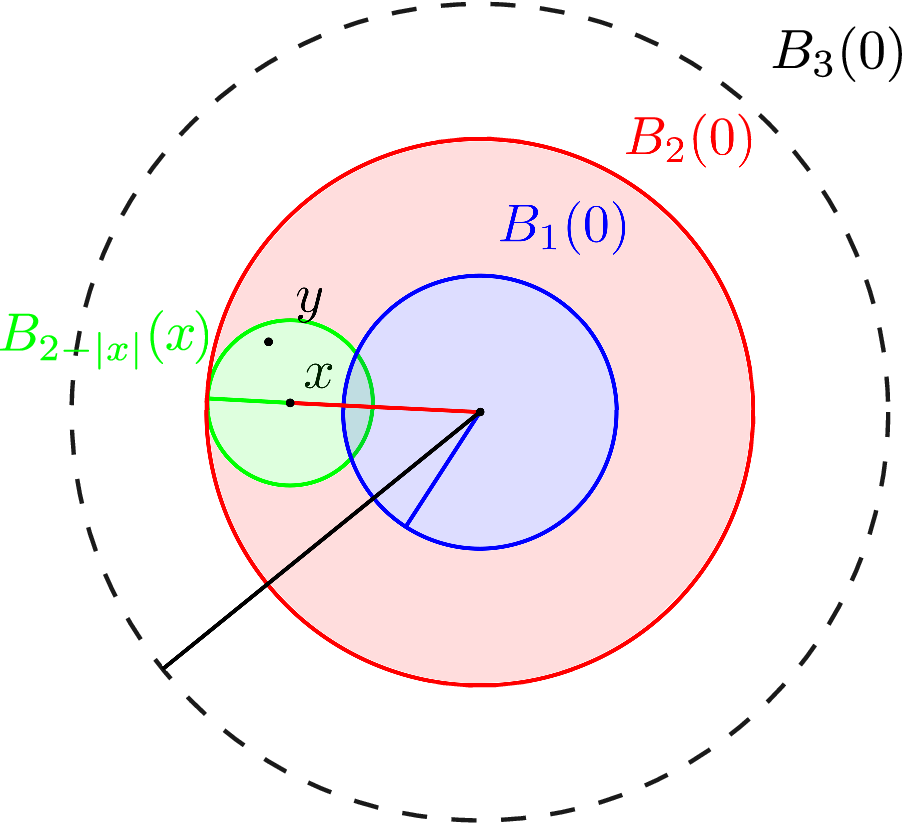}
    \caption{Schematic for the proof of the Evans-Krylov $C^{2,\alpha}$ estimate.}
    \label{ballsest}
\end{figure}

\begin{proof}
We follow \cite{HeinTosatti2020}, which is inspired by Chen-Wang \cite{chen2014c2alphaestimatemongeampereequationsholdercontinuous} and, in turn, by Simon \cite{Simon}. Figure \ref{ballsest} will help keep track of where the estimates are being applied.

It is enough to prove that
\begin{equation*}
    \sup_{x\in B_2(0)}\mu(\omega,x)\leq C',
\end{equation*}
where we define
\begin{align*}
    \mu(\omega, x) &= 
    (2-|x|)^\alpha \sup_{\substack{x\neq y\in B_{2-|x|}(x)}}\frac{|\omega(x)-\omega(y)|_{\omega_{\rm Euc}}}{|x-y|^\alpha}\\
    &= (2-|x|)^\alpha \lambda(\omega,x)^\alpha.
\end{align*}
Note, the case $x,y\in B_1(0)$ with $|x-y|\geq 2-|x|$ is trivial by examining the H\"older difference quotient directly. Now, suppose by way of contradiction that for all $i \in \mathbb{N}$ there exists some $\omega_i$ on $B_3(0)$ such that $$\omega_i^n=\omega_{\mathrm{Euc}}^n, \quad
\frac1C\omega_{\mathrm{Euc}}\leq \omega_i\leq C \omega_{\mathrm{Euc}},$$
and some $x_i\in B_2(0)$ such that 
\begin{equation*}
\mu(\omega_i,x_i)\geq \frac12 \sup_{x\in B_2(0)} \mu(\omega_i, x)
\end{equation*}
and such that $\mu(\omega_i,x_i), \lambda(\omega_i,x_i)$ go to $ \infty$ as $i \to \infty$. Let $y_i\in \overline{B_{2-|x_i|}(x_i)}$ realize the supremum in the definition of $\lambda_i = \lambda(\omega_i,x_i)$. A Taylor expansion of $\omega_i$ about $x_i$ immediately shows that $y_i \neq x_i$. 

We now pull back all objects (to the ``hat'' picture) under the map 
\begin{equation*}
\Phi_i: \mathbb{C}^n \to \mathbb{C}^n, \quad z\mapsto \frac{1}{\lambda_i} z.
\end{equation*} 
For the metrics, we also need to rescale appropriately; thus, we have in these cases: 
\begin{equation*}
\hat \omega_{\mathrm{Euc}}=\lambda^2_i\Phi_i^*(\omega_{\mathrm{Euc}})=\omega_{\mathrm{Euc}}, \quad \hat \omega_i = \lambda^2_i \Phi_i^*(\omega_i).
\end{equation*} 
On the ball $B_{3\lambda_i}(0)$ we still have the two properties
\begin{equation*} 
    \frac 1C  \hat\omega_{\mathrm{Euc}} \leq \hat \omega_i\leq C \hat \omega_{\rm Euc}, \quad \hat \omega_i^n = \hat\omega^n_{\mathrm{Euc}}. 
\end{equation*}
Letting $\hat x_i = \Phi_i^{-1}(x_i) \in B_{2\lambda_i}(0)$ and $\hat y_i = \Phi_i^{-1}(y_i) \in \overline{B_{2\lambda_i-|\hat x_i|}(\hat x_i)}\setminus\{\hat x_i\}$, we have 
\begin{equation}\label{eq:contra}
    \sup_{\hat x_i\neq \hat y\in B_{2\lambda_i -|\hat x_i|}(\hat x_i)}\frac{|\hat\omega_i(\hat x_i)-\hat \omega_i(\hat y)|_{\hat{\omega}_{\rm Euc}}}{|\hat x_i-\hat{y}|^\alpha} = 1, 
\end{equation}
and $\hat{y}_i$ realizes the supremum, by construction. Since the numerator of the difference quotient is bounded, we must have $|\hat{x}_i - \hat{y}_i| \leq C$. Moreover, 
\begin{equation*} 
    \mu(\omega_i, x_i) = (2\lambda_i-|\hat x_i|)^\alpha \sup_{\hat x_i\neq \hat y\in B_{2\lambda_i-|\hat x_i|}(\hat x_i)} \frac{|\hat \omega_i(\hat x_i)-\hat \omega_i(\hat y)|_{\hat\omega_{\rm Euc}}}{|\hat x_i-\hat y|^\alpha} 
\end{equation*}
Now, as $i\to \infty$, the left-hand side blows up, while on the right-hand side the $\sup$ is equal to $1$ by construction. Hence, it must be the case that $2\lambda_i-|\hat x_i|\to \infty$. That is, passing to a pointed limit centred at $\hat x_i$, the boundary of the green ball moves away to infinity, but $\hat y_i$ does not. We now prove the following claim:

\medskip

\noindent \textbf{Claim 1.} For all $R$, there exists $C_R$ independent of $i$, such that for all $i$ we have 
\begin{equation*} 
 \frac{|\hat\omega_i(\hat x)-\hat \omega_i(\hat y)|_{\hat\omega_{\rm Euc}}}{|\hat x-\hat y|^\alpha} \leq C_R, \quad \forall \hat x \neq \hat y\in B _R(\hat x_i). 
\end{equation*}

\begin{proof}[Proof of Claim 1]
    For any $\hat{x} \in B_{2\lambda_i}(0)$ we have that 
\begin{equation*} 
    \mu(\hat \omega_i, \hat x)\leq \sup_{\underline{\hat x}\in B_{2\lambda_i}(0)}\mu(\hat \omega_i, \underline{\hat x}) \leq 2\mu(\hat\omega_{i}, \hat x_i) = 2(2\lambda_i - |\hat{x}_i|)^\alpha \lambda(\hat\omega_i,\hat{x}_i). 
\end{equation*}
By assumption, $\lambda(\hat\omega_i,\hat{x}_i) = 1$. Thus, by the definition of $\mu$, we have 
\begin{align*}
    \sup_{\hat x\neq \hat y\in B_{2\lambda_i-|\hat x|}(\hat x)} \frac{|\hat \omega_i(\hat x)-\hat \omega_i(\hat y)|}{|\hat x-\hat y|^\alpha}
    \leq 2\left(\frac{2\lambda_i-|\hat x_i|}{ 2\lambda_i-|\hat x|}\right)^\alpha.    
\end{align*} 
If $\hat x, \hat y\in B_R(\hat x_i)$ for a fixed $R$, then 
\begin{equation*} 
    \hat y\in B_{2\lambda_i-|\hat x|}(\hat x), \quad \frac{2\lambda_i-|\hat x_i|}{2\lambda_i-|\hat x|}\leq C_R,
\end{equation*}
for all $i$ large enough depending on $R$. So the claim is proved for those $i$, and for the others we just need to maximize the Hölder norm of finitely many smooth functions over a fixed closed ball (all depending on $R$).
\end{proof} 

To finish, we require the following $C^k$ estimates of the metrics $\omega_i$: 

\medskip

\noindent \textbf{Claim 2.} For all $R,k$ there exists a constant $C_{R,k}<\infty$ such that for all $i$,
\begin{equation*}
\|\hat \omega_i\|_{C^k(B_R(\hat x_i))}\leq C_{R,k}. 
\end{equation*}

\begin{proof}[Proof of Claim 2]
The $\partial\overline{\partial}$-lemma with estimates says that there exists $C_R<\infty$ such that for $\omega$ a smooth, closed real (1,1)-form on $B_{2R}(\hat x_i)$, we have the existence of 
\begin{equation*}
u_{i,R}\in C^\infty(B_{2R}(\hat x_i)), \qquad \sqrt{-1}\partial\bar \partial u_{i,R}=\hat \omega_i\quad \text{on }B_{2R}(\hat x_i).
\end{equation*}
Furthermore, the smooth function $u_{i,R}$ obeys the following $L^\infty$-bound: 
\begin{equation*}
\|u_{i,R}\|_{L^\infty(B_R(\hat x_i))}\leq C_R\|\hat\omega_i\|_{L^\infty(B_{2R}(\hat x_i))} \leq C_R.
\end{equation*}
Tracing the definition of $u_{i,R}$ with respect to $\hat\omega_{\rm Euc}$ yields that 
\begin{equation*}
\Delta_{\hat\omega_{\mathrm{Euc}}} u_{i,R}=\mathrm{Tr}_{\hat\omega_{\mathrm{Euc}}}(\hat \omega_i).
\end{equation*}
By Claim 1, the right-hand side is bounded in $C^{0, \alpha}(B_R(\hat x_i))$ with bound independent of $i$. Then Schauder theory of the operator $\Delta_{\hat\omega_{\rm Euc}}$ gives 
\begin{equation}\label{eq:aux:est} 
\|u_{i,R}\|_{C^{2,\alpha}(B_R(\hat x_i))}\leq C_R. 
\end{equation}
Differentiating the Monge-Ampère equation yields the following: 
\begin{equation*} 
(\sqrt{-1}\partial\bar \partial u_{i,R})^n =\hat \omega_{\mathrm{Euc}}^n  \implies  \Delta_{\hat\omega_i} (Du_{i,R})=0
\end{equation*}
for any directional derivative $D$ in $\mathbb{C}^n$. The coefficients of the operator $\Delta_{\hat\omega_i}$ are uniformly bounded in $C^{0,\alpha}(B_R(\hat x_i))$ by Claim 1. The functions $Du_{i,R}$ are uniformly bounded in $C^{1,\alpha}(B_R(\hat{x}_i))$, so, in particular, in $L^\infty(B_R(\hat{x}_i))$, by Equation \eqref{eq:aux:est}. Thus, by Schauder theory of the operator $\Delta_{\hat\omega_i}$, the functions $Du_{i,R}$ are uniformly bounded in $C^{2,\alpha}(B_R(\hat{x}_i))$. Thus, $\hat{u}_{i,R}$ is itself uniformly bounded in $C^{3,\alpha}(B_R(\hat{x}_i))$ and we have gained one order. Iterating this argument gives the claim.
\end{proof}

After translating $\hat{x}_i$ to the origin, then, by Arzel\`a-Ascoli and diagonalization with respect to $R$, there exists a $C^\infty_{\rm loc}$ sublimit $\hat \omega_\infty$ on $\mathbb{C}^n$ satisfying the assumptions of Yau's Liouville Theorem \ref{cor:YauLT}. Thus, $\hat\omega_\infty$ is a Euclidean metric, that is to say, it is constant. On the other hand, by Equation \eqref{eq:contra} and the definition of $\hat{y}_i$,
\begin{equation}\label{eq:almost:contra}|\hat\omega_i(\hat{x}_i) - \hat\omega_i(\hat{y}_i)| = |\hat{x}_i - \hat{y}_i|^\alpha.\end{equation}
We also know that $0 < |\hat{x}_i - \hat{y}_i| \leq C$. Thus, it remains to establish a uniform positive lower bound on $|\hat{x}_i - \hat{y}_i|$ because then we can pass $\hat{x}_i - \hat{y}_i$ to a (non-zero) limit and conclude from Equation \eqref{eq:almost:contra} that $\hat\omega_\infty$ is not constant. However, due to the case $k = 1$ of Claim 2 above, and because $|\hat{x}_i - \hat{y}_i| \leq C$, we have that
$$|\hat{x}_i - \hat{y}_i|^\alpha = |\hat\omega_i(\hat{x}_i) - \hat\omega_i(\hat{y}_i)| \leq C|\hat{x}_i - \hat{y}_i|,$$
so the desired lower bound of $|\hat{x}_i - \hat{y}_i|$ follows because $\alpha < 1$.
\end{proof}

Your \textbf{fourth exercise} is to prove the $\partial\overline{\partial}$-lemma with estimates.

\subsection{Liouville on Calabi-Yau cylinders}\label{sec:LTonCyl}

Let $(Y,\omega_Y)$ be a compact simply-connected Calabi-Yau manifold, for instance, a K3 surface. We now wish to sketch a proof of the following: 

\begin{theorem}[Hein \cite{Hein}]\label{thm:HJH}
If on $\mathbb{C}^n\times Y$ we have a Calabi-Yau metric $\omega$ satisfying 
\begin{equation*}
    \frac1C\omega_{\rm Prod}\leq \omega \leq C\omega_{\rm Prod},\quad \text{where}\;\,\omega_{\rm Prod}=\omega_{\mathrm{Euc}}+\omega_Y,
\end{equation*} 
and if $[\omega]=[\omega_Y]$ in $H^2(\mathbb{C}^n\times Y, \mathbb{R})=H^2(Y,\mathbb{R})$, then, for some $A \in {\rm GL}(n,\mathbb{C})$,
\begin{equation*}
    \omega =A^*\omega_{\mathrm{Euc}}+\omega_Y.
\end{equation*}
\end{theorem}

The main idea here is to consider 
\begin{align*}
    n &= \sum_{j=1}^n dz_j \left(\frac{\partial}{\partial z_j}\right) \leq \left(\sum_{j=1}^n |dz_j|^2_\omega\right)^{1/2} \left(\sum_{j=1}^n \left|\frac{\partial}{\partial z_j}\right|^2_\omega\right)^{1/2}.    
\end{align*}
Both the two sums in parentheses are in $L^\infty(\mathbb{C}^n \times Y)$ by our assumption of uniform equivalence of $\omega$ and $\omega_{\rm Prod}$. They are also both subharmonic with respect to $\omega$, by the Bochner formulas for holomorphic $1$-forms and vector fields.

\medskip

\noindent \textbf{Goal.} Show that these two terms converge to their Euclidean values at $\infty$. Then the above inequality will read $n \leq \sqrt{n} \cdot \sqrt{n}$. Thus, equality will have to hold at every step, in particular, in the Bochner formulas. Thus $\nabla_\omega^{\rm LC}(dz_j)=0$, $\nabla_\omega^{\rm LC}(\partial/\partial z_j) = 0$, which forces $\omega$ to split as an isometric product.

\medskip

To implement this, we form the tangent cone at $\infty$, that is, for all $\lambda>0$ we let 
\begin{equation*} 
    \omega_\lambda=\lambda^{-2} \Phi_\lambda^*\omega, \qquad \Phi_\lambda(z,y) = (\lambda z,y). 
\end{equation*}
Then $\omega_\lambda$ is uniformly equivalent to $\omega_{\mathrm{Euc}}+\lambda^{-2}\omega_Y$, so we may take a weak limit of currents as $\lambda\to \infty$ to get a Calabi-Yau metric $\omega_\infty$ on $\mathbb{C}^n$ uniformly equivalent to $\omega_{\mathrm{Euc}}$. Thus, by Yau's Liouville Theorem, for some $A \in {\rm GL}(n,\mathbb{C})$,
\begin{equation*} 
\omega_\infty =\text{w-}\lim_{\lambda\to \infty}\omega_\lambda = A^*\omega_{\mathrm{Euc}}.
\end{equation*}
 Without loss, $A=\mathrm{Id}$. The difficulty is to prove that $|\partial/\partial z_j|^2_{\omega_\lambda} \to 1, |dz_j|^2_{\omega_\lambda} \to 1$ pointwise a.e. as $\lambda \to \infty$. Here the subharmonicity of these functions enters. It allows us to obtain pointwise a.e. convergence from the weak convergence $\omega_\lambda \to \omega_\infty$ via a heat flow smoothing and uniform heat kernel estimates. For the functions $|dz_j|^2_{\omega_\lambda}$ this is a little delicate and we only achieve this after summation in $j$.

\begin{remark}
    We have argued by reduction to Yau's Liouville Theorem \ref{cor:YauLT}. In \cite{lilizhang}, Li-Li-Zhang show that similar arguments actually yield a new proof of that theorem, which we have presented in Section \ref{sec:YauLT}. We do not know whether, conversely, Theorem \ref{thm:HJH} can be proved using only classical estimates as in Section \ref{sec:YauLTOld}.
\end{remark}

\subsection{From Liouville to Evans-Krylov on cylinders}\label{sec:fromLTtoKEonCyl}

Let $\overline{\mathscr{Y}}$ be a compact Calabi-Yau manifold, fibered over some compact Kähler space $\overline{B}$ with generic fiber a compact Calabi-Yau manifold $Y$ as in the previous section. We can drop the simply-connected assumption on $Y$ as this adds only minor technicalities to the proof. However, we assume that $\dim Y > 0$ and $\dim \overline{B} > 0$ to avoid trivialities. Fix K\"ahler classes $\mathbf{k}_{\overline{B}}, \mathbf{k}_{\overline{\mathscr{Y}}}$ on $\overline{B},\overline{\mathscr{Y}}$, then introduce the class
\begin{equation*} 
    \mathbf{k}_t = \mathbf{k}_{\overline B}+t^2\mathbf{k}_{\overline{\mathscr{Y}}} 
\end{equation*}
on $\overline{\mathscr{Y}}$, which is K\"ahler for all $t > 0$ and allows us to get a unique Calabi-Yau metric $\omega_t \in \mathbf{k}_t$ by Yau. The goal is to understand $\omega_t$ in the \emph{adiabatic limit} $t \to 0$.

Let $\mathscr{Y}$ be a sufficiently small tubular neighborhood of some regular fiber. Then $\mathscr{Y}$ has an induced regular fibration over a ball $B = B_3(0) \subset \mathbb{C}^n$. Trivialize $\mathscr Y$ as a $C^\infty$-fiber bundle $\mathscr Y\cong B\times Y$. Introduce non-Hermitian reference metrics 
\begin{equation*} 
g_{z,t}=g_{\mathrm{Euc}}+t^2g_{\mathscr{Y}_z}
\end{equation*}
for all $z\in B$ such that $g_{\mathscr{Y}_z}$ is the unique Calabi-Yau metric on $\mathscr{Y}_{z}$, the fiber over $z$, in the class $\mathbf{k}_{\overline{\mathscr{Y}}}|_{\mathscr{Y}_z}$. From Tosatti \cite{Tosatti2010}, there is a uniform constant $C < \infty$ such that 
\begin{equation*} 
\frac1C g_{z,t}\leq g_t \leq Cg_{z,t} 
\end{equation*}
for all $z \in B$ and $0 < t \leq 1$. This plays the role of a $C^2$ estimate in our setting.

\begin{theorem}[Hein-Tosatti \cite{HeinTosatti2020}]\label{thm:HT1}
For all $\alpha\in (0,1)$ there is a constant $C'=C'(\alpha, C)$ such that for all $0 < t \leq 1$ we have that
\begin{equation}\label{eq:HTSemiNorm}
\sup_{\substack{x=(z,y)\in B_{1}(0)\times Y}}\;\sup_{x'\in B^{g_{z,t}}(x,1)}\frac{|g_t (x)-\mathbb{P}^{g_{z,t}}_{x'x}(g_t(x'))|_{g_{z,t}}}{d^{g_{z,t}}(x, x')^\alpha}\leq C',
\end{equation}
where $\mathbb{P}^{g_{z,t}}_{x'x}$ denotes $g_{z,t}$-parallel transport along the minimal $g_{z,t}$-geodesic from $x'$ to $x$, assuming such a geodesic exists and is unique.
\end{theorem} 

The proof proceeds via blowups to find a contradiction with Theorem \ref{thm:HJH}. No direct proof is known, and classical $C^3$ estimate computations seem to yield only partial results \cite{TosattiWeinkoveYang}. Modulo Theorem \ref{thm:HJH}, the main difficulties in the proof are:
\begin{enumerate}
    \item Coming up with the appropriate $C^{\alpha}$ seminorm \eqref{eq:HTSemiNorm}.
    \item Getting Schauder estimates that are uniform in $0< t \leq 1$.
    \item Using notation as in Section \ref{sec:fromLTtoEK}, there are now three cases after blowing up: $\lambda_i t_i \to \infty$, $\lambda_i t_i \to c \in (0,\infty)$ and $\lambda_i t_i \to 0$. In the (bad) third case, where the blown-up metrics collapse to the flat metric on $\mathbb{C}^n$, we need to ensure that the nontrivial difference quotient of $\hat \omega_i$ at $\hat{x}_i$ and $\hat{y}_i$ essentially comes from the base-base components of $\hat\omega_i$; else there will be no contradiction in the limit.
\end{enumerate}

Theorem \ref{thm:HT1} can be extended to higher orders to obtain an asymptotic expansion. For $z \in B$ let $A_\mu \in \Omega^{0,1}(\mathscr{Y}_z, T^{1,0}\mathscr{Y}_z)$ denote the \emph{Kodaira-Spencer form} (harmonic with respect to $\omega_{\mathscr{Y}_z}$) describing the variation of complex structure of the fibers in direction $\partial/\partial z_\mu$. The \emph{Weil-Petersson form} $\omega_{\rm WP}$ is a semipositive closed $(1,1)$-form on $B$ whose $\mu\bar{\nu}$-component at a point $z \in B$ is the $L^2$-inner product of $A_\mu$ and $A_{\nu}$ with respect to $\omega_{\mathscr{Y}_z}$. Let $\omega_{\overline B}$ be the unique member of $\mathbf{k}_{\overline B}$ such that ${\rm Ric}_{\omega_{\overline B}} = \omega_{\rm WP}$ weakly and globally on the base. Lastly, let $\omega_{\rm SRF}$ be any \emph{semi-Ricci-flat form} on $\mathscr{Y}$, i.e. a closed real $(1,1)$-form  such that $\omega_{\rm SRF}|_{\mathscr{Y}_z} = \omega_{\mathscr{Y}_z}$ for all $z \in B$.

\begin{theorem}[Hein-Tosatti \cite {HeinTosatti}]\label{thm:HT2} After shrinking $B$ if needed, we have that
\begin{equation*} 
\omega_t = \omega_{\overline{B}}+t^2 \omega_{\mathrm{SRF}}+t^4\sqrt{-1}\partial\bar \partial \left(g_{\overline B}^{\bar \mu\nu}\Delta^{-2}_{\omega_{\mathscr{Y}_z}}\left\langle A_\nu,A_{\bar\mu}\right\rangle_{\omega_{\mathscr{Y}_z}}^\circ \right)+\cdots
\end{equation*}
as $t \to 0$ on $\mathscr{Y}$, where $^\circ$ denotes the average-zero part of a function on $\mathscr{Y}_z$ with respect to the Calabi-Yau metric $\omega_{\mathscr{Y}_z}$. The $\,\cdots$ part may include terms that are only $o(1)$ as $t \to 0$ but are pulled back from the base.
\end{theorem}

The proof of this theorem is very technical, although in addition to Theorem \ref{thm:HT1} it only involves linear analysis (mainly, Liouville theorems for harmonic functions on cylinders). The philosophy is to first prove uniform derivative estimates with respect to collapsing reference metrics by blowup and contradiction, and to then derive the existence of an expansion as a corollary. To get a sense of how such expansions are derived, your {\bf fifth exercise} will be to prove the following lemma:

\begin{lemma}\label{lem:1} Let $B,Y$ be smooth manifolds with metrics $g_B, g_Y$, where $Y$ is compact without boundary. Then for all $\ell\in \mathbb{N}$ there is an $A<\infty$, and for all $u\in C^\infty(B\times Y)$ there is a $u_0\in C^\infty(B)$, such that for all $t > 0$ and $C < \infty$:
\begin{equation*}
|\nabla^\ell_{g_{B}+t^2g_{Y}}(u)|_{g_B+t^2g_Y}\leq C \implies |u-u_0|\leq ACt^\ell.
\end{equation*}
\end{lemma}

Furthermore, a parabolic version of Theorem \ref{thm:HT2} was shown in \cite{HeinLeeTosatti}. Recall that a holomorphic line bundle on a compact Kähler manifold $\overline{\mathscr{Y}}$ is \emph{nef} if its Chern class lies in the closure of the Kähler cone. The Abundance Conjecture predicts that if the canonical bundle of $\overline{\mathscr{Y}}$ is nef, it must be \emph{semiample}, i.e. sufficiently high tensor powers provide a well-defined map from $\overline{\mathscr{Y}}$ to projective space, the \emph{Iitaka fibration} of $\overline{\mathscr{Y}}$, whose image has dimension equal to the \emph{Kodaira dimension} of $\overline{\mathscr{Y}}$. We also recall that the canonical bundle is nef if and only if all solutions to Kähler-Ricci flow on $\overline{\mathscr{Y}}$ are immortal, i.e. they exist for all positive times \cite{TianZhang}. In this definition one can equivalently use either the unnormalized Kähler-Ricci flow $\partial\omega/\partial t = -{\rm Ric}_\omega$ or the normalized flow $\partial\omega/\partial t = -{\rm Ric}_\omega - \omega$. The parabolic version of Theorem \ref{thm:HT2} in \cite{HeinLeeTosatti} implies the following result, which settles a conjecture of Song-Tian \cite{SongTian}. For some additional, very recent developments in this direction, see also \cite{Wenrui}.

\begin{theorem}[Hein-Lee-Tosatti \cite{HeinLeeTosatti}]\label{thm:HLT}
Let $\overline{\mathscr{Y}}$ be a compact Kähler manifold whose canonical bundle is nef and semiample, with Kodaira dimension strictly between $0$ and $\dim \overline{\mathscr{Y}}$. Then, for any normalized Kähler-Ricci flow on $\overline{\mathscr{Y}}$, the Ricci curvature stays bounded as $t \to \infty$ on any regular tube $\mathscr{Y}$ of the Iitaka fibration.
\end{theorem}
 
\section{Calabi-Yau cones and Klemmensen's thesis}\label{sec:cones}

In Sections \ref{sec:LTonCyl}--\ref{sec:fromLTtoKEonCyl} we considered Calabi-Yau metrics on collapsing fibrations. In this setting one can prove a $C^2$ estimate using classical computations \cite{Tosatti2010} but such computations seem to fail to yield a $C^3$ estimate or a Liouville theorem on $\mathbb{C}^n \times Y$. Thus, we chose an indirect approach, proving Liouville geometrically and deducing a $C^{2,\alpha}$ estimate by blowup and contradiction. However, it seems possible that the classical $C^3$ computation can be made to work after all because the conclusion of Theorem \ref{thm:HJH} is still $\nabla_{\omega_{\rm Ref}}\omega = 0$. We now consider a setting where the $C^2$ estimate is unknown and the conclusion of Liouville is no longer that $\nabla_{\omega_{\rm Ref}}\omega = 0$.

\subsection{Setting and statement of Klemmensen's theorem}

\begin{definition}
Let $(L, g_L)$ denote a compact Riemannian manifold.
\begin{enumerate}
    \item The \textbf{cone} $(\mathscr{C}, g_\mathscr{C})$ over $L$ is the product space $\mathscr{C} = \mathbb{R}^+\times L$ with metric 
    \begin{equation*}
    g_\mathscr{C}=dr^2\oplus r^2 g_L.
    \end{equation*}
    \item We say $(\mathscr{C},g_{\mathscr{C}})$ is a \textbf{Calabi-Yau cone} if there is a $g_{\mathscr{C}}$-parallel complex structure $J$ and a $g_\mathscr{C}$-parallel $J$-holomorphic volume form $\Omega$ on the manifold $\mathscr{C}$. These are not unique and are part of the data of a Calabi-Yau cone. Note that ${\rm Ric}_{g_{\mathscr{C}}}=0$.
\end{enumerate}  
\end{definition}

\begin{remark}
    Even if one only assumes that $(\mathscr{C},g_{\mathscr{C}})$ as in item 1 above is Kähler with respect to some complex structure $J$, one can already deduce from this that 
    $$\omega_{\mathscr{C}} = \frac{\sqrt{-1}}{2}\partial\bar \partial r^2.$$
\end{remark}

\begin{example}\label{ex:CYcones}
\begin{enumerate}
    \item Flat cones $\mathbb{C}^n/\Gamma$, where $\Gamma \subset \mathrm{SU}(n)$ is finite and acts freely on $S^{2n-1}$, are Calabi-Yau cones. The metric $\omega_\mathscr{C}$ is the pushdown of $\omega_{\mathrm{Euc}}$ on $\mathbb{C}^n$. Similarly, the holomorphic volume form $\Omega$ is the pushdown of $\Omega_{\rm Euc} = dz^1 \wedge \cdots \wedge dz^n$ on $\mathbb{C}^n$. The standard example of such a finite group $\Gamma$ is
    \begin{equation*}
    \Gamma = \left\langle
    {\rm Diag}\left(e^{\frac{2\pi i}{n}}, e^{\frac{2\pi i}{n}}, \ldots, e^{\frac{2\pi i}{n}}\right)
    \right\rangle\cong \mathbb{Z}/n\mathbb{Z}.  
    \end{equation*}
    This cone admits a \emph{crepant resolution} by the total space of the holomorphic line bundle
    $\mathscr{O}_{\mathbb{P}^{n-1}}(-n).$
    See Example \ref{blocki} for an example of how this arises for $n=2$.
    \item Consider the \emph{conifold} or \emph{Morse singularity} or \emph{Stenzel cone}
    \begin{equation*}
    \mathscr{C} \cup \{0\}=\{z^2_1+z^2_2+\cdots + z^2_{n+1}=0\}\subset \mathbb{C}^{n+1}, \;\, \omega_{\mathscr{C}} = \frac{\sqrt{-1}}{2}\partial\bar \partial r^2,  \;\, r=|z|^{\frac{n-1}{n}}.
    \end{equation*} 
    The case $n=2$ is the same as above, whereas $|{\rm Rm}_{g_\mathscr{C}}|_{g_\mathscr{C}}\sim r^{-2}$ for $n \geq 3$. There can be no upper or lower sectional curvature bounds because $\mathrm{Ric}_{g_\mathscr{C}}=0$.  
\end{enumerate}
\end{example}

\begin{figure}[h]
    \includegraphics[scale=.6]{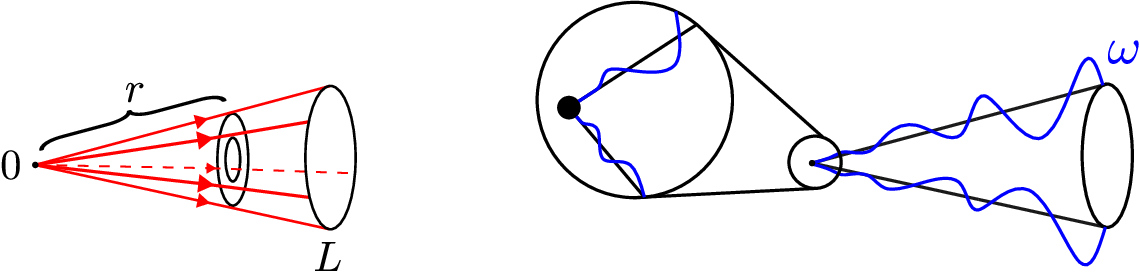}
    \caption{A Calabi-Yau cone. We hope that if we consider a metric $\omega$ uniformly equivalent to $\omega_{\mathscr{C}}$, then, if $\omega$ has bounded scalar curvature, it must approach the cone at the apex, as opposed to oscillating an infinite number of times (which is allowed by the uniform equivalence).}
   \label{cone}
\end{figure}

\begin{theorem}[Klemmensen \cite{klemmensen}]\label{thm:JJK}
For every Calabi-Yau cone $\mathscr{C}$ there exists an $\alpha >0$ such that, given any two constants $C,D<\infty$, there exists some constant $C'<\infty$ such that if $\omega$ is a K\"ahler metric on $B_3(0)\subset \mathscr{C}$ satisfying 
\begin{equation*}
    \frac1C\omega_\mathscr{C}\leq \omega\leq C\omega_\mathscr{C}, \quad |\mathrm{Scal}_\omega|\leq D,
\end{equation*} 
then there exists an automorphism $\Phi$ of $(\mathscr{C},J)$ with 
$ \Phi_*(r\partial_r)=r\partial_r$ such that
\begin{equation*}
    |\Phi^*\omega-\omega_\mathscr{C}|_{\omega_\mathscr{C}}\leq C' r^\alpha \quad \text{on}\;\,B_1(0).
\end{equation*}
\end{theorem}

\begin{remark}
    \begin{enumerate}
    \item The automorphisms $\Phi$ of $(\mathscr{C},J)$ that preserve the Euler vector field $r \partial_r$ are the so-called \emph{transverse automorphisms} of the cone. They form a finite-dimensional complex Lie group. At this point it is helpful to recall that the full automorphism group of $\mathbb{C}^n$, $n \geq 2$, is infinite-dimensional:
    \begin{equation*} 
        \Phi(z_1,z_2) = (z_1,z_2+f(z_1)) 
    \end{equation*}
    is a (volume-preserving!) automorphism of $\mathbb{C}^2$ for any arbitrary holomorphic function $f:\mathbb{C}\to \mathbb{C}$.
    However, the transverse automorphism group of $\mathbb{C}^n$ with respect to its standard flat cone structure is simply ${\rm GL}(n,\mathbb{C})$.
    
    \item For $\mathscr{C}=\mathbb{C}^n$, Theorem \ref{thm:JJK} follows from the usual de Giorgi-Nash-Moser theory and Evans-Krylov estimates applied to the equation
    \begin{equation*}
        |\mathrm{Scal}_\omega|=\left|\Delta_{\omega}\log\left(\frac{\omega^n}{\omega^n_{\mathrm{Euc}}}\right)\right|\leq D.
    \end{equation*}
    \item Theorem \ref{thm:JJK} is similar, but not equivalent to theorems of Hein-Sun \cite{HeinSun} and Chiu-Sz\'ekelyhidi \cite{ChiuGabor}. The key difference is that Klemmensen's proof is elementary but he has to assume $\omega$ is uniformly equivalent to $\omega_{\mathscr{C}}$, whereas the other authors make no such assumption (in their setting) but need to use the highly nontrivial Donaldson-Sun theory of tangent cones of Kähler-Einstein spaces \cite{DS}.
\end{enumerate}
\end{remark}

As one might expect at this point, the proof of Theorem \ref{thm:JJK} factors into two steps: a Liouville theorem over the Calabi-Yau cone $\mathscr{C}$, and a proof of an Evans-Krylov type estimate (that implies Theorem \ref{thm:JJK}) by blowup and contradiction to Liouville.

\subsection{Liouville on Calabi-Yau cones}

\begin{theorem}[Klemmensen \cite{klemmensen}]\label{thm:JJKLiou} If $\omega$ is a Kähler metric on $\mathscr{C}$ with 
\begin{equation*}
\frac 1C\omega_\mathscr{C}\leq \omega\leq C\omega_\mathscr{C}, \quad \mathrm{Scal}_\omega=0,
\end{equation*}
then $\omega=\Phi^*\omega_\mathscr{C}$ for some transverse automorphism $\Phi$ of $\mathscr{C}$.
\end{theorem}

\begin{remark}
\begin{enumerate}
    \item If $\omega_{\mathscr{C}}$ is not flat and if $\Phi$ is not a scaling map, $(r,\theta) \mapsto (\lambda r, \theta)$, then $\Phi^*\omega_{\mathscr{C}}$ is not parallel with respect to $\omega_{\mathscr{C}}$. The conifold (Example \ref{ex:CYcones}(2)) provides such an example. This shows that it may not be so simple to prove Theorem \ref{thm:JJKLiou} using classical derivative estimates as in Section \ref{sec:YauLTOld}.
    \item $\mathscr{C} = \mathbb{C}^n$ is allowed, so we get a third proof of Yau's Liouville Theorem \ref{cor:YauLT} after Sections \ref{sec:YauLTOld} and \ref{sec:YauLT}. However, during the proof we will need to use Theorem \ref{thm:EKclass} in a crucial way, which directly implies Theorem \ref{cor:YauLT} via scaling.
\end{enumerate}
\end{remark}

\begin{proof}[Sketch proof of Theorem \ref{thm:JJKLiou}]
We first reduce to the Ricci-flat case. We have that 
    \begin{equation*}
    0 = \mathrm{Scal}_\omega = {\rm Tr}_\omega({\rm Ric}_\omega - {\rm Ric}_{\omega_{\mathscr{C}}}) = - \Delta_\omega \log\left(\frac{\omega^n}{\omega^n_{\mathscr{C}}}\right) 
    \end{equation*}
    because $\omega_{\mathscr{C}}$ is Ricci-flat. Thus, the log volume ratio is a bounded entire $\omega$-harmonic function on $\mathscr{C}$. Moser's Harnack inequality, which holds on $(\mathscr{C},\omega_{\mathscr{C}})$, implies that bounded entire $\omega_{\mathscr{C}}$-harmonic functions on $\mathscr{C}$ are constant. However, the Harnack inequality is stable under uniform equivalence of metrics (for complete manifolds this is due to Saloff-Coste \cite{SaloffCoste}, and the proof carries over to our setting). Thus, the same result holds on $(\mathscr{C},\omega)$, so $\omega^n$ is a multiple of $\omega_{\mathscr{C}}^n$ and, hence, ${\rm Ric}_\omega = 0$.
    
    It is a well-known theme that, among manifolds of nonnegative Ricci curvature, cones are characterized by the constancy of a number of monotone quantities. For technical reasons, the most convenient quantity in our setting is Perelman's entropy \cite{Perelman}, as adapted to the static case by Ni \cite{LeiNi}. Perelman's monotonicity formula says that if $u>0$ is a solution to the heat equation on $(\mathscr{C},\omega)$,
    \begin{equation*}
    \int_{\mathscr{C}} u\, dV_{\omega} = 1, \quad u =(4\pi t)^{-n}e^{-f},
    \end{equation*} then the entropy, defined as  
    \begin{equation*} 
    \mathscr{W}(t) = \int_{\mathscr{C}}\left( t|\nabla_\omega f|_\omega^2 +f-2n\right)u\, dV_\omega, 
    \end{equation*}
    decreases. More precisely, one can show that 
    \begin{equation*} 
    \mathscr{W}'(t)= -2t\int_\mathscr{C} \left|\nabla^2_\omega f-\frac{g}{2t}\right|_\omega^2 u\, dV_\omega\leq 0.
    \end{equation*}
    Furthermore, this integral vanishes if and only if $g$ is a cone metric. It is a healthy {\bf exercise} to verify this formula (try this without looking up \cite{LeiNi} first). 

    \medskip

    A small caveat is again that we are working on an incomplete manifold, so some care is needed to justify integrating by parts etc. However, again using methods due to Saloff-Coste \cite{SaloffCoste}, we may construct such a solution $u$ with $f$ comparable to $r^2/t$ uniformly on $\mathscr{C} \times \mathbb{R}^+$. This behavior suffices to make everything rigorous. 
    
    Our intuition going forward is that $\mathscr{W}(t)$ interpolates between the entropies $\mathscr{W}_0$ of the tangent cone of $(\mathscr{C},\omega)$ at the apex and $\mathscr{W}_\infty$ of the tangent cone of $(\mathscr{C},\omega)$ at infinity. See Figure \ref{entropy} below for a visual. However, at this point, all we know for sure is that $\mathscr{W}(t)$ actually has a limit $\mathscr{W}_0, \mathscr{W}_\infty$ as $t \to 0,\infty$, respectively. This is due to the monotonicity of $\mathscr{W}$ and its boundedness, which follow from $f \sim r^2/t$.

    \begin{figure}[b]
    \sidecaption
    \includegraphics[width=0.645\linewidth]{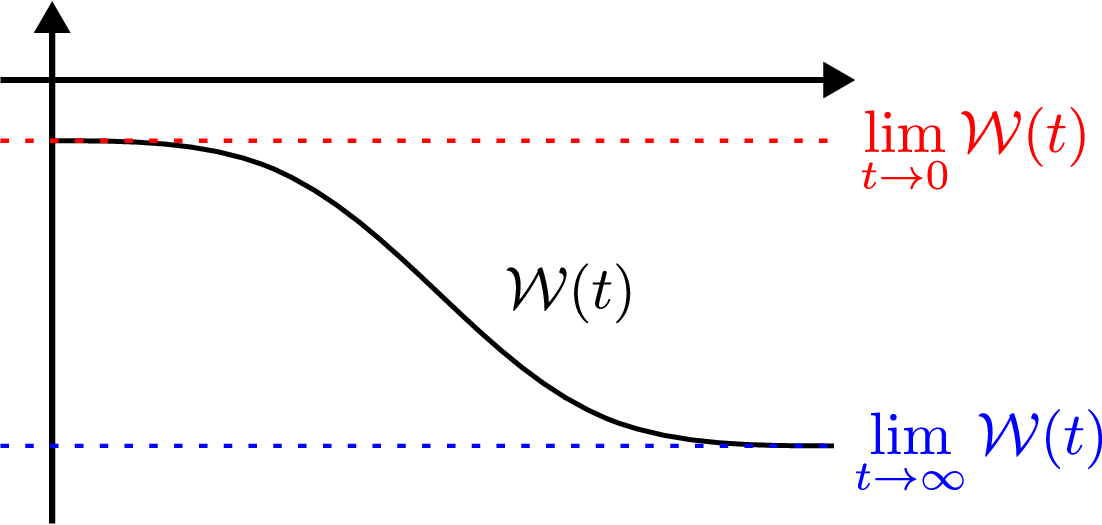}
    \caption{The entropy $\mathscr{W}(t)$.}
    \label{entropy}
    \end{figure}
    
  To proceed, for all $\lambda>0$, we define $\omega_\lambda = \lambda^{-2}\Phi_\lambda^*\omega$,
    where $\Phi_\lambda (r, \theta) = (\lambda r, \theta)$. Note that $\lambda^{-2}\Phi_\lambda^*\omega_\mathscr{C}=\omega_\mathscr{C}$, and hence we have the properties
    \begin{equation}\label{eq:saviors}
    \omega_{\lambda}^n = \omega_{\mathscr{C}}^n, \quad \frac1C\omega_\mathscr{C}\leq \omega_\lambda  \leq C \omega_\mathscr{C},
    \end{equation}
    uniformly in $\lambda$. Furthermore, we define
    \begin{equation*} 
    u_\lambda(x,t) = \lambda^{2n} u(\Phi_\lambda(x), \lambda^2t) 
    \end{equation*}
    for all $(x,t) \in \mathscr{C}\times\mathbb{R}^+$.
    Then $u_\lambda$ solves the heat equation with respect to $\omega_\lambda$, and
    $$\int_{\mathscr{C}} u_\lambda \, dV_{\omega_\lambda} = 1, \quad u_\lambda = (4\pi t)^{-n}e^{-f_\lambda},$$
    where $f_\lambda$ is comparable to $r^2/t$ uniformly on $\mathscr{C} \times \mathbb{R}^+$ as well as with respect to $\lambda$. We can again define an entropy $\mathscr{W}_\lambda$ of $(\mathscr{C},\omega_\lambda)$ using $u_\lambda$. Then, by construction,
    \begin{equation} \label{eq:new:savior}
    \mathscr{W}_\lambda(t) = \mathscr{W}(\lambda^2t).
    \end{equation}
    Taking $\lambda \to 0$ and $\lambda \to \infty$, we can first of all prove that the metrics $\omega_\lambda$ converge locally smoothly along a subsequence due to the properties \eqref{eq:saviors}, the local regularity of the complex Monge-Ampère equation (Theorem \ref{thm:EKclass} and Claim 2 in its proof) and the Arzelà-Ascoli theorem. The same holds for the heat equation solutions $u_\lambda$ thanks to the uniform control $f_\lambda \sim r^2/t$. Hence we obtain limit metrics $\omega_0,\omega_\infty$ satisfying \eqref{eq:saviors}, together with heat equation solutions $u_0,u_\infty$ satisfying $f_0,f_\infty \sim r^2/t$, and with associated entropy functions $\mathscr{W}_0(t), \mathscr{W}_\infty(t)$, which are {\bf constant} in $t$ due to \eqref{eq:new:savior}.

    \medskip
    
    \noindent \textbf{Goal.} Find transverse automorphisms $\Phi_{0, \infty}$ such that $\omega_{0, \infty} = \Phi^*_{0, \infty}\omega_{\mathscr{C}}$.

    \medskip Note that, because of the constancy of $\mathscr{W}_0(t), \mathscr{W}_\infty(t)$, we do know that $\omega_{0,\infty}$ are Ricci-flat Kähler {\bf cone} metrics on $(\mathscr{C},J)$, uniformly equivalent to $\omega_{\mathscr{C}}$. So our goal is exactly the special case of Theorem \ref{thm:JJKLiou} where $\omega$ is in addition assumed to be a cone metric. Thus, this goal is clearly necessary for proving the theorem. But it is also sufficient: The constancy of $\mathscr{W}_i(t)$ shows that $f_i = r_i^2/4t + c_i$ ($i = 0,\infty$), where $r_i$ denotes the radius function of the cone metric $\omega_i$ and $c_i \in \mathbb{R}$ is determined by the requirement that $\int_{\mathscr{C}} u_i \,dV_{\omega_i} = 1$. This in turn implies that
    \begin{equation*} 
    \mathscr{W}_{i}=\log\left(\frac{\mathrm{Vol}(L,g_{L,i})}{\mathrm{Vol}(S^{2n-1})}\right).
    \end{equation*} 
    But if we know that the cones $(\mathscr{C},\omega_0)$ and $(\mathscr{C},\omega_\infty)$ are isometric, then their links have the same volume, so $\mathscr{W}_0 = \mathscr{W}_\infty$, so $\mathscr{W}(t)$ is constant in $t$, so $\omega$ is a cone metric, and the same proof as for $\omega_0,\omega_\infty$ will show that $\omega = \Phi^*\omega_{\mathscr{C}}$, as desired.

    \medskip

    \noindent {\bf Proof of Goal.} We write out the argument only for $\omega_0$. Since this is a Kähler cone metric on $(\mathscr{C}, J)$, its Reeb vector field $\xi_{0} = J(r_0\partial_{r_0})$ is a holomorphic Killing field with respect to $\omega_0$. It is of linear growth with respect to $\omega_0$, so, due to the uniform equivalence of $\omega_0$ and $\omega_{\mathscr{C}}$, with respect to $\omega_{\mathscr{C}}$ as well. 
    A Liouville type theorem of Hein-Sun \cite{HeinSun} shows that, for any holomorphic vector field $\xi$ on $(\mathscr{C},J)$ which has linear growth with respect to $\omega_{\mathscr{C}}$, we must have that $\xi\in \mathfrak{p}\oplus J\mathfrak{p}$, where 
    \begin{equation*}
    \mathfrak{p}=\mathrm{Span}\left\{ \begin{array}{c}
        r\partial_r, \text{ gradients of }\xi_{\mathscr{C}}\text{-invariant }\\\Delta_{\omega_\mathscr C}\text{-harmonic functions that are }\\\text{homogeneous  of degree }2\text{ w.r.t. }r\partial_r
        \end{array}\right\}, \quad \xi_{{\mathscr{C}}} = J(r \partial_r).
    \end{equation*} 
    Furthermore, $\mathfrak{p}\oplus J\mathfrak{p}$ is the Lie algebra of the transverse automorphism group. Since $\xi_0$ is a Killing field with respect to $\omega_0$, hence only has compact orbits, one can also show that $\xi_0 \in J\mathfrak{p}$. Then the uniqueness theorem of Martelli-Sparks-Yau \cite{marinellisparksyau} implies that $\xi_0 = \xi_{{\mathscr{C}}}$ because in the Lie algebra $J\mathfrak{p}$ of transverse isometries of $\omega_{\mathscr{C}}$, the Reeb vector field $\xi_{{\mathscr{C}}}$ is the only Reeb vector field of a Calabi-Yau cone metric on $(\mathscr{C},J)$. Then, lastly, the Sasakian version due to Nitta-Sekiya \cite{nittasekiya} of the Bando-Mabuchi uniqueness theorem for Kähler-Einstein metrics \cite{bandomabuchi} implies that $\omega_0 = \Phi_0^*\omega_\mathscr{C}$.
\end{proof}

\begin{remark}
\begin{enumerate}
\item Chen-Wang \cite{chen2014longtimebehaviourconical} proved a Yau type Liouville theorem over divisorial cones $\mathbb{C}^{n-1} \times \mathbb{C}_\beta$. This was also done by passing to the tangent cones at an apex and at infinity. However, because the tangent cones are flat in this case, classical direct estimates are sufficient for their proof. 
\item Conceptually, all ingredients of the proof of Theorem \ref{thm:JJKLiou} (except for the reduction from ${\rm Scal}_\omega=0$ to ${\rm Ric}_\omega=0$) can be found in Donaldson-Sun theory \cite{DS, HeinSun}; only their implementation is different. The point of Theorem \ref{thm:JJKLiou} is to demonstrate that the assumption $\frac{1}{C}\omega_{\mathscr{C}} \leq \omega \leq C\omega_{\mathscr{C}}$, i.e. the $C^2$ estimate in the theory of the complex Monge-Ampère equation, is sufficient to exorcize all considerations of Gromov-Hausdorff limits, Cheeger-Colding theory and $K$-stability of valuations from the Donaldson-Sun theory (in our situation). Instead, only the elementary $C^\infty$ differential-geometric versions of these ideas are required.
\end{enumerate}
\end{remark}

\subsection{From Liouville to Evans-Krylov on cones}

The final step in the proof of Theorem \ref{thm:JJK} is to define a suitable H\"older seminorm, and to prove it is finite (and actually bounded by $C$ and $D$) via blowup and contradiction to Theorem \ref{thm:JJKLiou}. This is the technical heart of Klemmensen's paper \cite{klemmensen}. It would be too complicated to go into the details. We only mention that the definition of a suitable $C^\alpha$ type seminorm is one of the most important points, and is inspired by the equivalent reformulation of the standard Euclidean $C^\alpha$ seminorm as a polynomial approximation seminorm in Krylov's book \cite[Section 3.3]{Krylov1996LecturesOE}. More concretely, for $\rho, \nu \in \mathbb{R}^+$ and $x \in B_1(0) \subset \mathscr{C}$ define
$$B_{\rho,\nu}(x) := (B_\rho(x) \cap B_1(0))\setminus B_\nu(x) \subset \mathscr{C},$$
where all balls are defined using $\omega_{\mathscr{C}}$. Then define
\begin{align*}
[\omega]_{\alpha,f} :=\sup_{\rho,\nu \in \mathbb{R}^+} \sup_{x \in B_1(0)} \frac{f({\rm dist}_{\omega_{\mathscr{C}}}(0,B_{\rho,\nu}(x)))}{\rho^\alpha}\inf_{\pi \in \Sigma_{3C}}\inf_{\eta \in \Sigma_{\rm loc}} \sup_{B_{\rho,\nu}(x)}|\omega-\pi-\eta|_{\omega_{\mathscr{C}}}.
\end{align*}
Here, $f: [0,1] \to [0,1]$ is a weight function depending on a parameter $\delta \in (0,1)$ which vanishes on $[0,\delta]$ and converges to $1$ as $\delta \to 0$. This weight ensures that the seminorm is finite. Then the contradiction argument proves boundedness in terms of $C,D$ of this seminorm, uniformly as $\delta \to 0$. Moreover,
$$\Sigma_{3C} = \{\pi = \Phi^*\omega_{\mathscr{C}}: \Phi\;\,\text{transverse automorphism}, \;\,\frac{1}{3C}\omega_{\mathscr{C}} \leq \pi \leq 3C \omega_{\mathscr{C}}\}.$$
Lastly, $\Sigma_{\rm loc}$ is a set of local $(1,1)$-forms comprising all constant $(1,1)$-forms in the charts of a sufficiently fine Whitney covering of $\mathscr{C}$ by holomorphic coordinate balls. If $B_{\rho,\nu}(x)$ is not contained in any ball of the covering, only $\eta = 0$ is allowed.

\section{Open problems}

\begin{enumerate}
    \item Consider a Bedford-Taylor weak solution to $\det(u_{j\bar k}) = 1$ on a ball $B \subset \mathbb{C}^n$. Is there always a point $z \in B$ such that $u$ is smooth in a neighborhood of $z$?
    \item Can the Liouville type Theorem \ref{thm:HJH} or the Evans-Krylov type Theorem \ref{thm:HT1} be proved using classical $C^3$ computations after all? (This seems unlikely for Klemmensen’s results because their statements are too different from the classical case.)
    \item Can the higher-order asymptotics of \cite{HeinTosatti,HeinLeeTosatti} (e.g. Theorem \ref{thm:HT2}) be proved using microlocal analysis if the Evans-Krylov type Theorem \ref{thm:HT1} is assumed? How about the higher-order \textit{estimates} (rather than just asymptotics) of \cite{HeinTosatti,HeinLeeTosatti}?
    \item Do immortal solutions to normalized K\"ahler-Ricci flow $\partial\omega/\partial t = -\text{Ric}(\omega) - \omega$ on compact Kähler manifolds have globally bounded Ricci curvature (not just away from the singular Iitaka fibers, as in Theorem \ref{thm:HLT})? At least on surfaces?
    \item Does Klemmensen’s Theorem \ref{thm:JJK} still hold on cones that are only scalar-flat rather than Ricci-flat?
    \item Can the $C^2$ estimate assumption in Klemmensen’s Theorem \ref{thm:JJK} be verified without using any advanced tools from Donaldson-Sun theory?
    \item Is something like Yau’s Liouville Theorem \ref{cor:YauLT} true on any noncompact Calabi-Yau manifolds other than cylinders or cones? A very recent result is \cite{Oussama}.
    \item What about \textit{cusps} as singularity models rather than cylinders or cones? A good example are negative K\"ahler-Einstein metrics on line bundles over $K3$ surfaces, which have unbounded curvature. See \cite[Section 2]{fu2021asymptoticskahlereinsteinmetricscomplex} for details.
\end{enumerate}

\section{Solutions to the exercises}\label{sec:solutions}

\subsection*{Exercise 1: Weak solutions of the complex Monge-Amp\`ere equation}

We first recall the definition of a weakly psh function on a domain $V \subset \mathbb{C}^n$.

\begin{definition}
We define $\mathrm{PSH}(V)$ as the class of all upper semicontinuous functions $u: V \to \mathbb{R} \cup \{-\infty\}$ satisfying one of the following two equivalent conditions:
\begin{enumerate}
    \item For all complex lines $\ell \subset \mathbb{C}^n$, the restriction $u|_{\ell \cap V}$ is weakly subharmonic, i.e. it satisfies the sub-mean inequality on all closed disks contained in $\ell \cap V$.
    \item Either $u \equiv -\infty$, or $u$ is locally integrable and the distributional derivative $dd^c u$, which defines a closed current, is in addition \emph{positive}, i.e. $\int dd^c u \wedge \eta \geq 0$ for all test forms $\eta$ that are \emph{strongly positive}, i.e. for every $z \in V$ we can write
$$\eta|_z = \sum_j \bigwedge_k \sqrt{-1}\alpha_{j,k} \wedge \overline{\alpha_{j,k}}, \quad \alpha_{j,k}\in\Lambda^{1,0}(\mathbb{C}^n).$$
\end{enumerate}
\end{definition}

\vspace{-1mm}

\noindent Let us prove that $u(z) = (1 + |z_1|^2)|z_2| \in \mathrm{PSH}(B)$ for $B = B_1(0) \subset \mathbb{C}^2$. By direct computation, away from $\{z_2=0\}$ the complex Hessian of $u$ is given by
\begin{equation}\label{eq:hessian:of:example}
    (u_{j\bar k})=\begin{pmatrix}
    |z_2|&\dfrac{\bar{z}_1 z_2}{2|z_2|}\\[0,3cm]
    \dfrac{z_1\bar z_2}{2|z_2|}&\dfrac{1+|z_1|^2}{4|z_2|}
\end{pmatrix}.
\end{equation}
This is nonnegative definite. For a smooth function this property is equivalent to the function being psh. Now we apply the following extension lemma:

\begin{lemma}\label{lem:ext}
    Let $F=\{\varphi=-\infty\}$ for some $\varphi\in \mathrm{PSH}(V)$. Assume that $u\in \mathrm{PSH}(V\setminus F)$ is bounded. Then $u^*(z) := \limsup_{y\to z} u(y)$ belongs to $\mathrm{PSH}(V)$. 
\end{lemma}

\noindent In our case we can use $\varphi(z) = \log |z_2|$. Moreover, $u^* = u$ because $u \in C^0(B)$.

\vspace{1mm}

\begin{definition}
    For $u \in L^\infty_{\rm loc}(V) \cap {\rm PSH}(V)$ we can define the complex Monge-Amp\`ere operator $(dd^c u)^n$ in two equivalent ways:
    \begin{enumerate}
        \item If a sequence $u_i \in C^\infty_{\rm loc}(V)\cap \mathrm{PSH}(V)$ decreases pointwise to $u$, then the $(dd^c u_i)^n$ converge weakly as Radon measures and we define $(dd^c u)^n$ to be their limit.
        \item For $k = 2, \ldots, n$, if $(dd^c u)^{k-1}$ is already defined as a current, and is known to be closed and positive, we define $(dd^c u)^k$ as a current via $$\int (dd^c u)^k \wedge \eta = \int u \cdot (dd^c u)^{k-1} \wedge dd^c \eta$$ for all test forms $\eta$. This makes sense because, as a positive current, $(dd^c u)^{k-1}$ is a differential form whose coefficients are complex Radon measures, hence can be multiplied by $L^\infty_{\rm loc}$ functions. And $(dd^c u)^k$ is again closed and positive.
    \end{enumerate}
\end{definition}

\noindent We apply this to our example:

\begin{enumerate}
    \item To proceed by approximation, let
$u_\varepsilon(z)=(1+|z_1|^2)(|z_2|^2+\varepsilon)^{1/2}$.
Then
\begin{align*}
(u_{\varepsilon,j\bar{k}}) = \begin{pmatrix} (|z_2|^2+\varepsilon)^{1/2} & \dfrac{\bar z_1 z_2}{2(|z_2|^2+\varepsilon)^{1/2}} \\[0,4cm]
\dfrac{z_1\bar z_2}{2(|z_2|^2+\varepsilon)^{1/2}} & \dfrac{(1+|z_1|^2)(|z_2|^2+2\varepsilon)}{4(|z_2|^2+\varepsilon)^{3/2}} \end{pmatrix},
\end{align*}
which is positive definite with
$$\det(u_{\varepsilon,j\bar{k}}) =\frac{
|z_2|^2+2\varepsilon(1+|z_1|^2)
}
{
4\bigl(|z_2|^2+\varepsilon\bigr)
}.
$$
As $\varepsilon \to 0$, this function converges to $1/4$ pointwise on $B \setminus \{z_2 = 0\}$. Also,
$$|{\det(u_{\varepsilon,j\bar{k}})}|\leq \frac{1}{4} + \frac{1}{2}(1 + |z_1|^2)$$
for all $\varepsilon > 0$. Thus, by the dominated convergence theorem,
$$\int \eta (dd^c u_\varepsilon)^2 = \int 8\eta \,{\det(u_{\varepsilon,j\bar{k}})} \,dV_{\rm Euc} \to \int 2\eta\,dV_{\rm Euc}$$
for all continuous functions $\eta$ with compact support in $B$. So $(dd^c u)^2 = 2 dV_{\rm Euc}$ as Radon measures, or equivalently $\det(u_{j\bar{k}}) = 1/4$ in the weak sense.
\item We may also compute $(dd^c u)^2$ in the sense of currents. Fix a smooth function $\eta$ with compact support in $B$. By definition,
$$\int \eta (dd^c u)^2 = \int u \cdot dd^c u \wedge dd^c \eta.$$
By a quick calculation, \eqref{eq:hessian:of:example} holds in the sense of distributions on all of $B$, so the coefficients of $dd^c u$ are in fact $L^1_{\rm loc}$ functions. Thus,
\begin{align*}\int \eta (dd^c u)^2 = \lim_{\varepsilon \to 0} \int_{|z_2| \geq \varepsilon} u \cdot 4(u_{1\bar{1}}\eta_{2\bar{2}} + u_{2\bar{2}}\eta_{1\bar{1}} - u_{1\bar{2}}\eta_{2\bar{1}} - u_{2\bar{1}}\eta_{1\bar{2}})\,dV_{\rm Euc}.
\end{align*}
We apply the divergence theorem twice to remove the derivatives from $\eta$. This yields the expected volume integral $\int 2\eta\,dV_{\rm Euc}$, plus a possible residue
\begin{align*}
 \lim_{\varepsilon \to 0} \frac{2}{\varepsilon}\int_{|z_2|=\varepsilon} {\rm Re}[(u_2 u_{1\bar{1}} - u_1u_{2\bar{1}})\eta z_2 - (\eta_{\bar{2}}u_{1\bar{1}} - \eta_{\bar{1}}u_{1\bar{2}})u\bar{z}_2 ] \,dA_{\rm Euc}.\end{align*}
This vanishes by \eqref{eq:hessian:of:example}. Only a term $u_2 u_{2\bar{2}}$ could have caused trouble.
\end{enumerate}

\noindent Lastly, we note that there is also an abstract way of computing $(dd^c u)^2$: as a Radon measure, the Monge-Ampère operator of a bounded psh function never charges any pluripolar sets $\{\varphi=-\infty\}$ as in Lemma \ref{lem:ext}.
Thus, $(dd^c u)^2$ is simply the pushforward of $2dV_{\rm Euc}$ from $B\setminus \{z_2 = 0\}$ to $B$ under the inclusion, which is again $2dV_{\rm Euc}$.

\medskip

\noindent {\bf Geometric interpretation.} We continue our analysis of Blocki's \cite{blocki1999regularity} weak solution $u(z) = (1+|z_1|^2)|z_2|$ to the complex Monge-Ampère equation. As it turns out, we can interpret the singular Calabi-Yau metric $\omega_u = dd^c u$ on $\mathbb{C}^2$ as the pullback of a flat Kähler metric on $\mathbb{C}^2/\{\pm {\rm Id}\}$ under the crepant resolution of this singularity. To see this, we first recall the well-known isomorphism
\begin{align*}
\mathbb{C}^2/\{\pm {\rm Id}\} \to \mathscr{X} = \{x_1x_2 = x_3^2\} \subset \mathbb{C}^3, \quad [(w_1,w_2)] \mapsto (w_1^2,w_2^2,w_1w_2).
\end{align*}
Then we consider the diagram
\begin{align*} 
\mathscr{Y}=\overline{\pi^{-1}(\mathscr{X}\setminus \{0\})}&\;\,\subset\;\,  \mathrm{Bl}_0(\mathbb{C}^3)\\
\downarrow\quad&\qquad\;\;\; \downarrow\pi\\
\mathscr{X}\;\;\;&\;\,\subset\quad\;\mathbb{C}^3.
\end{align*}
We can show that $\mathscr{Y} \cong \mathscr{O}_{\mathbb{P}^1}(-2) = K_{\mathbb{P}^1}$, which implies that $K_{\mathscr{Y}}$ is trivial and hence that the resolution is crepant. To see the isomorphism (and that it maps $\mathscr{Y} \cap \pi^{-1}(0)$ to the zero section of the line bundle), notice that
\begin{align*}\mathscr{Y} = {\rm Bl}_0(\mathscr{X}) &\cong {\rm Bl}_0(\mathbb{C}^2/\{\pm{\rm Id}\}) \\
&= {\rm Bl}_0(\mathbb{C}^2)/\{\pm{\rm Id}\} = \mathscr{O}_{\mathbb{P}^1}(-1)/\{\pm {\rm Id}\} = \mathscr{O}_{\mathbb{P}^1}(-2).\end{align*}

To verify our claim about the solution $u$, we start with the potential $|w_1|^2 + |w_2|^2 $ of a flat Kähler metric on $\mathbb{C}^2$. This clearly pushes forward to the function $|x_1| + |x_2|$ on $\mathscr{X}$. We can parametrize an open subset of $\mathscr{Y}$ by coordinates $(z_1,z_2)$ via 
$$(z_1,z_2) \mapsto ([z_1^2: 1: z_1],(z_1^2z_2, z_2, z_1z_2)) \in {\rm Bl}_0(\mathbb{C}^3) \subset \mathbb{P}^2 \times \mathbb{C}^3.$$
Here we have realized ${\rm Bl}_0(\mathbb{C}^3)$ as the variety $\{(\ell, p) \in \mathbb{P}^2 \times \mathbb{C}^3: p \in \ell\}$ and $\pi$ is the projection onto the second factor. Thus, in our chart $(z_1,z_2)$, the exceptional divisor $\mathscr{Y} \cap \pi^{-1}(0)$ is cut out by $z_2 = 0$. Moreover, the flat Kähler potential $|x_1| + |x_2|$ on $\mathscr{X}$ lifts to the Blocki example $u(z) = (1 + |z_1|^2)|z_2|$ on $\mathscr{Y}$, as claimed.

\subsection*{Exercise 2: Details of the classical \texorpdfstring{$C^2$}{C^2} and \texorpdfstring{$C^3$}{C^3} estimates} 

\noindent (i) Fix $p \in B$. Since $H = (u_{j\bar{k}}(p))$ is a positive definite Hermitian matrix, 
    by the spectral theorem there exists a unitary matrix $Q$ such that 
    \begin{equation*} 
    Q^\top H \bar{Q} = D = \mathrm{Diag}(\lambda_1, \ldots, \lambda_n), \quad \lambda_1,\ldots,\lambda_n > 0.
    \end{equation*}
    Write $z=Qw$. By the chain rule, we see
    \begin{align} \label{eq:trf}
    \begin{split}\frac{\partial u}{\partial \bar w_k}&=\frac{\partial \bar z_m}{\partial \bar w_k} \frac{\partial u}{\partial \bar z_m},\\
    \quad \frac{\partial^2 u}{\partial w_j \partial \bar w_k} &= \frac{\partial z_\ell}{\partial w_j} \frac{\partial^2 u}{\partial z_\ell \partial \bar z_m} \frac{\partial \bar z_m}{\partial \bar w_k} = (Q^\top H \bar Q)_{j\bar k} = \lambda_j \delta_{j\bar{k}}. 
    \end{split}\end{align}
    However, notice that the rotation $Q$ depends on the chosen point $p$. This is why, in the $C^2$ estimate calculation, we consider the rotationally invariant quantity
    $$\Delta u(z) = {\rm Tr}(H) = {\rm Tr}(\bar{Q} D\bar{Q}^{-1}) = {\rm Tr}(D) = \Delta u(w)$$ instead of any individual second derivative $u_{\ell \bar{\ell}}$.

\medskip

     \noindent (ii)
    Consider two local holomorphic coordinate systems $z_i$ and $w_\alpha$ on $\mathbb{C}^n$. Recall the formula for the Christoffel symbols of a Kähler metric:
    \begin{equation*}
    \Gamma_{jk}^i = \frac{\partial g_{j\bar \ell}}{\partial z_k}g^{\bar \ell i}.
\end{equation*}
Using this and the metric transformation laws
    $$g_{\alpha\bar\beta}= \frac{\partial z_i}{\partial w_\alpha } g_{i\bar{j}} \frac{\partial \bar{z}_j}{\partial \bar{w}_\beta }, \quad g^{\bar\alpha \beta} = \frac{\partial \bar{w}_\alpha}{\partial \bar{z}_i} g^{\bar{i}j}\frac{\partial w_\beta}{\partial z_j},$$
    we can check, with a little bit of work but without surprises, that
    \begin{equation} \label{eq:Gamma:trafo}
    \Gamma^\alpha_{\beta\gamma} = \frac{\partial w_\alpha}{\partial z_k}\frac{\partial z_i}{\partial w_\beta}\frac{\partial z_j}{\partial w_\gamma}\Gamma^k_{ij} + \frac{\partial w_\alpha}{\partial z_\ell}\frac{\partial^2 z_\ell}{\partial w_\beta \partial w_\gamma}.
    \end{equation}
    Thus, for any two Kähler metrics $\omega, \tilde{\omega}$ on $\mathbb{C}^n$, the differential operator
    \begin{equation*}
    \nabla^{\rm LC}_{\omega}-\nabla^{\rm LC}_{\tilde\omega}: \Gamma(T^{1,0}\mathbb{C}^n) \times \Gamma(T^{1,0}\mathbb{C}^n) \to \Gamma(T^{1,0}\mathbb{C}^n)  
    \end{equation*}
    transforms like a $(1,2)$-tensor under holomorphic coordinate changes: the Hessian terms in \eqref{eq:Gamma:trafo} are the same for $\omega$ and $\tilde\omega$, so they cancel out in the difference.
   
    \medskip
    
    \noindent (iii), (iv) These are standard facts about Bochner formulas but you may not have seen them. A good reference for this, particularly with a view towards applications in Riemanian geometry, is \cite[Chapters 1.3, 3.3, 4.6, 5.4 and 6]{Ballmann}. The notations and facts that we require are as follows. Let $(M,g)$ be a Riemannian manifold.
    \begin{enumerate}
    \item For a tensor $T \in \Gamma(T^{p,q}M)$ we have its covariant derivative $\nabla T \in \Gamma(T^{p,q+1}M)$ defined in the standard way using the Levi-Civita connection of $g$. We also have the formal $L^2$ adjoint of the operator $\nabla$ given by
    \begin{equation*} 
    \nabla^* T = -  {\rm Tr}_{12}(\nabla T) = -\sum_i (\nabla_{E_i} T)(E_i^\flat,\dots) \in \Gamma(T^{p-1,q}M)
    \end{equation*}
    for any orthonormal frame $E_i$ and its dual coframe $E_i^\flat$. One often calls $-\nabla^*T$ the \emph{divergence} of $T$, particularly when $T$ is a vector field. On the other hand, when $T=\alpha$ is an alternating $q$-form ($p=0$), then $\nabla^*T$ coincides with \begin{equation*} d^* \alpha = (-1)^{nq+1}\star d\star \alpha, \quad \alpha\wedge \star \beta =\left\langle \alpha,\beta \right\rangle dV, \quad n = \dim M.\end{equation*} 
    
    \item For $T \in \Gamma(T^{p,q}M)$ we define the \emph{rough Laplacian} $\Delta T\in\Gamma(T^{p,q}M)$ to be
    \begin{equation} \label{eq:dumbo}
    \Delta T = -\nabla^*\nabla T = \mathrm{Tr}_{12}(\nabla^2 T) = \sum_i (\nabla_{E_i} (\nabla_{E_i} T)-\nabla_{\nabla_{E_i} E_i}T). 
    \end{equation}
    This immediately specializes to the standard Laplacian of the components of $T$ when $M$ is flat Euclidean space.
    It is also straightforward to prove that \begin{equation*} 
    \frac{1}{2}\Delta|T|^2= |\nabla T|^2 +\left\langle \Delta T,T \right\rangle.
    \end{equation*}

    \item Bochner(-Weitzenböck) formulas relate operators such as the \emph{Laplace-Beltrami operator} $\Delta_d = dd^* + d^*d$ on forms, or $\Delta_\partial, \Delta_{\overline\partial}$ in the Kähler case, to the rough Laplacian $\Delta$. In fact, such formulas explain why $\Delta_d$ etc. are called Laplacians in the first place. The formulas required in this course are:
    \begin{itemize}
    \item[a)] $\Delta_d\alpha = \nabla^*\nabla \alpha + {\rm Ric}(\alpha)$ holds for all $1$-forms $\alpha$, where ${\rm Ric}(\alpha)^\sharp = {\rm Ric}(\alpha^\sharp)$. If $\alpha = dh$ for a holomorphic function $h$ on a Kähler manifold, then 
    $$\Delta_d \alpha = d\Delta_d h = d(2 \Delta_{\overline\partial}h) = d(2\overline\partial^*\overline\partial h) = 0.$$
    Thus, in particular, if ${\rm Ric} \geq 0$, then $\Delta |\alpha|^2 \geq 2|\nabla \alpha|^2$ by \eqref{eq:dumbo}.
    \item[b)] $2\overline\partial^*\overline\partial\xi = \nabla^*\nabla\xi - {\rm Ric}(\xi)$ holds for all $(1,0)$-vector fields $\xi$ in the Kähler case. Thus, if $\xi$ is holomorphic and if ${\rm Ric} \leq 0$, then $\Delta|\xi|^2 \geq 2|\nabla \xi|^2$.
    \end{itemize}
    \end{enumerate}

    \noindent (v) At this point we are switching from the standard linear coordinates on $\mathbb{C}^n$ to holomorphic normal coordinates for the Kähler metric $\omega_u$, but we keep using the same indices for both sets of coordinates. So we need to check that the $C^2$ estimate computation \eqref{eq:C2:est} is still valid. This follows from the holomorphic transformation law \eqref{eq:trf} of the complex Hessian: in the new coordinates, we still have that
    $$\frac{\partial^2}{\partial z_\ell \partial \bar{z}_m} \log \det(u_{j\bar{k}}) = \frac{\partial^2}{\partial z_\ell \partial \bar{z}_m} \log |{\det DF}|^2 = 0,$$ where $F$ is the holomorphic transition map.

\subsection*{Exercise 3: Different faces of the Taub-NUT metric} 

Let $z_1, z_2$ be the standard complex coordinates on $\mathbb{C}^2$ and define the ansatz \begin{equation*} u(z_1,z_2) = (a^2+a^4)+(b^2+b^4),\end{equation*} where $a,b$ are determined by the equations $$|z_1|=ae^{a^2-b^2},\;\,|z_2|=be^{b^2-a^2}.$$
Our goal is to show that $\omega = i\partial \bar\partial u$ defines a Ricci-flat Kähler metric on $\mathbb{C}^2$ which is not uniformly comparable to $\omega_{\mathrm{Euc}}$.

Consider the real polar coordinates given by 
\begin{align*}
x_1 &=  \log|z_1| = \log a + a^2-b^2,\\
x_2 &= \log|z_2| = \log b + b^2-a^2.
\end{align*}
\noindent Since $\partial_{z_j}x_k=\delta_{jk}/(2z_j)$, we have $u_{z_j\bar z_k}=u_{x_jx_k}/(4z_j\bar z_k)$. Thus, the problem boils down to calculating the real Hessian of $u$ with respect to the coordinates $(x_1, x_2)$.
Substituting the Jacobian of $a,b$ with respect to $x_1,x_2$ and the expression of $u$ given by $a,b$, by straightforward computation we get that the real Hessian $(u_{x_ix_j})$ is
$$
\frac{4}{1+2(a^2+b^2)}
\begin{pmatrix}
a^2((1+2b^2)^2+4a^2b^2)&4a^2b^2(1+a^2+b^2)\\
4a^2b^2(1+a^2+b^2)&b^2((1+2 a^2)^2+4a^2b^2)
\end{pmatrix}.
$$
Thus, $\det(u_{x_jx_k})=16a^2b^2$ and $\det(u_{z_j\bar z_k})=1$. Since $u_{z_1\bar{z}_1} > 0$ and $\det(u_{z_j \bar{z}_k}) > 0$, 
we conclude that $\omega$ is a Kähler metric.

We now show that $\omega$ is not comparable to $\omega_{\mathbb C^2}$. The complex Hessian is
\begin{equation*}
\frac{1}{1+2(a^2+b^2)}
\begin{pmatrix}
e^{2(b^2-a^2)}((1+2b)^2+4a^2 b^2)&4 \bar{z}_1 z_2 (1+a^2+b^2)\\
4 \bar{z}_2 z_1 (1+a^2+b^2)&e^{2(a^2-b^2)}((1+2a)^2+4a^2 b^2)
\end{pmatrix}.
\end{equation*}
Let $\lambda_1,\lambda_2$ denote its eigenvalues. Then clearly
\begin{equation*}
\lambda_1 + \lambda_2=e^{2(b^2-a^2)} \frac{(1+2b^2)^2+4a^2b^2}{1+2(a^2+b^2)} + e^{2(a^2-b^2)} \frac{(1+2a^2)^2+4a^2b^2}{1+2(a^2+b^2)}.
\end{equation*}
If we fix $a=1$ and send $b\to\infty$, then
$\mathrm{Tr}_{\omega_{\rm Euc}}(\omega)\to \infty$. Thus, the metric $\omega$ is not equivalent to the Euclidean metric, nor can it be given by $A^*\omega_{\mathrm{Euc}}$, $A \in {\rm GL}(2,\mathbb{C})$.

The previous computations did not give much geometric information. However, we also have an $S^1$-symmetry, which we can use to recover the Gibbons-Hawking description of the Taub-NUT metric in \cite{GibbonsHawking,lebrun}. The action is given by
$$e^{it}\cdot(z_1,z_2)=(e^{it}z_1,e^{-it}z_2).$$
This action preserves both $\omega$ and the holomorphic symplectic form $\Omega=dz_1\wedge dz_2$. By our previous computation, $\omega^2 = \Omega\wedge \overline{\Omega}$, and from this one can directly check that $\omega$, $\mathrm{Re}\, \Omega$, $\mathrm{Im}\, \Omega$ satisfy the hyper-K\"ahler triple condition. We denote them by $\omega_1,\omega_2,\omega_3$ and define the complex structure $I_a$ by $\omega_a(X,Y)=g(I_aX,Y)$ for $a = 1,2,3$. 

The real infinitesimal generator of this $S^1$-action is
$$
\xi|_{(z_1,z_2)}=\left.\frac{d}{dt}\right|_{t=0}(e^{it} \cdot (z_1,z_2)) = i( z_1\partial_{z_1}-z_2\partial_{z_2} -\bar z_1\partial_{\bar z_1} +\bar z_2\partial_{\bar z_2})
$$
leading to three real moment-map components $\mu_a$ with respect to $\omega_a$. We package them as a real and a complex function 
$\mu_{\mathbb{R}} = \mu_1$, $\mu_{\mathbb{C}} = \mu_2 +i \mu_3$ on $\mathbb{C}^2$ satisfying
\begin{equation*}
d\mu_{\mathbb{R}} = -\xi \lrcorner \omega, \quad d\mu_{\mathbb{C}}
=-i \xi \lrcorner \Omega.
\end{equation*}
Since $u$ is invariant under the action of the standard torus $\mathrm{U}(1)^2 \subset \mathrm{U}(2)$, we can use polar coordinates to show that, up to an additive constant, $\mu_{\mathbb{R}}=\frac12 (|a|^2-|b|^2)$. For the complex moment map, straightforward computation gives $\mu_{\mathbb{C}}=z_1z_2$ up to an additive constant. The hyper-K\"ahler identities tell us that 
$$g(d\mu_a,d\mu_b) = g(I_a\xi, I_b\xi) = |\xi|^2\delta_{ab},$$ 
meaning that $d\mu_a$ and $\xi^\flat$ are an orthogonal coframe of $\mathbb{R}^4$ with respect to $g$.
Thus, introducing the function $V=|\xi|^{-2}$ and the connection $1$-form $\theta=V\xi^\flat = VI_1d\mu_1$ of the $S^1$-action, we recover the Gibbons-Hawking ansatz
$$g=V(d\mu_1^2+d\mu_2^2+d\mu_3^2)+V^{-1}\theta^2.$$
Lastly, defining the Euclidean radius $r = (\mu_1^2 + \mu_2^2 + \mu_3^2)^{1/2}$, we can compute that 
$$V = |\xi|^{-2} = 1+\frac{1}{2r},$$
and so we obtain the standard form of the Taub-NUT metric.

\subsection*{Exercise 4: The \texorpdfstring{$\partial\bar \partial$}{\partial\bar \partial}-lemma with estimates}

The statement needed in the proof of Theorem \ref{thm:EKclass} (Claim 2) is an interior estimate: a bounded closed $(1,1)$-form on a ball has a proportionally bounded potential on a smaller concentric ball. We prove a slightly stronger estimate up to the boundary.

\begin{lemma}
    For all $n$ there exists a constant $C$ such that the following holds. Let $\omega$ be a smooth $d$-closed real $(1,1)$-form on a neighborhood of the closed unit ball $\overline{B}\subset\mathbb C^n$. Then there exists a smooth function $u$ on $B$ such that 
\begin{equation*} \sqrt{-1} \partial\bar \partial u= \omega, \quad \|u\|_{L^\infty(B)}\leq C\|\omega \|_{L^\infty (B)}.\end{equation*}
\end{lemma}

\begin{proof}
First, we define the following real $1$-form:
\begin{align*}
    \eta|_z = \sqrt{-1}\sum_{j,k}
  \left(\int_0^1s \omega_{j\bar k}(sz)\,ds\right)
  (z_j\,d\bar z_k-\bar z_k\,dz_j).
\end{align*} 
This obviously satisfies 
\begin{align}\label{eq:eta_bound}
\|\eta\|_{L^\infty(B)}\leq C\|\omega\|_{L^\infty(B)}.
\end{align}
By a straightforward computation, using the closedness of $\omega$, we find $d\eta=\omega$.

The $(0,2)$-part of $d\eta=\omega$ is zero, so $\bar\partial\eta^{0,1}=0$, so we can apply Rudin's explicit solution operator \cite[Theorem~16.7.2]{Rudin} to
$f = \eta^{0,1}$ to obtain a function $v$ such that
\begin{equation}\label{eq:rudin-estimate}
  \bar\partial v=f, \quad
  \|v\|_{L^\infty(B)}
  +[v]_{C^{0,1/2}(B)}
  \leq C\|f\|_{L^\infty(B)}.
\end{equation}
Set $u=2\operatorname{Im}v.$
Then $\sqrt{-1}\partial\bar\partial u = \partial\bar\partial(v-\bar v) = \partial\eta^{0,1}+\bar\partial\eta^{1,0} = d\eta = \omega.$
The required bound follows from \eqref{eq:rudin-estimate} and \eqref{eq:eta_bound}.
\end{proof}

We briefly explain how to construct and estimate this $v$, following Rudin \cite{Rudin}. We start from an arbitrary solution to $\bar\partial q =f$, then normalize the ``holomorphic part'' of $q$ by subtracting the Cauchy transform of $q$,
$$v(z)=q(z)-C[q](z) = \int_{\partial B}\frac{q(z)-q(\xi)}{(1-\langle z,\xi\rangle)^n}\,dA(\xi).$$ 
Here $\langle\cdot,\cdot\rangle$ denotes the standard Hermitian inner product on $\mathbb{C}^n$ (conjugate linear in the second argument). Using Stokes, Rudin writes $v=g+E$,
$$g(z)=\frac1n\int_B\frac{\langle f(w),z-w\rangle}{(1-\langle z,w\rangle)^n(1-\langle w,z\rangle)}\,dV(w),$$
where $g$ extends to a $C^{0,1/2}$ function on $\overline B$ and $E(z)\to 0$ if $z\to\xi\in \partial B$. The key point here is that $g$ is explicit in terms of $f$ while the contributions of the arbitrary solution $q$ are contained in $E$, which vanishes at $\partial B$.
Thus, $v$ can be reconstructed from $g|_{\partial B}$ using the Bochner-Martinelli formula \cite[Section 16.5.8]{Rudin},
\begin{align*}
v(z) &=\int_{\partial B} \frac{1-\left\langle\xi, z\right\rangle}{|\xi-z|^{2n}}\left( \frac 1n \int_B \frac{\left\langle f(w), \xi - w\right\rangle}{(1-\left\langle\xi, w\right\rangle)^n (1-\left\langle w,\xi \right\rangle)}\, dV(w)\right) dA(\xi)  \\
&-\frac1n \int_B \frac{\left\langle f(w), w-z\right\rangle}{|w-z|^{2n}}\, dV(w).
\end{align*} 
Rudin then estimates the $C^{0,1/2}$ seminorm of each of the two summands separately; see the discussion of $J_1$ and $J_2$ on pp. 360--361 of \cite{Rudin}. The upshot is that
$$[v]_{C^{0,1/2}(B)}\leq C\|f\|_{L^\infty(B)}.$$
After subtracting a constant from $v$, the $L^\infty$ bound in \eqref{eq:rudin-estimate} follows from this.

\subsection*{Exercise 5: From derivative estimates to asymptotics (Lemma \ref{lem:1})}

The function $u_0$ will simply be the fiberwise average of $u$,
$$u_0(z)=\frac{1}{{\rm Vol}(Y,g_Y)}\int u(z,y)\, dV_{g_Y}(y).$$
For $\ell=0$ the desired estimate is then obvious from the triangle inequality. For $\ell = 1$ it follows by integrating the $y$-gradient of $u(z,y)$ along a minimal $g_Y$-geodesic from a point $y_0(z)$ with $u(z,y_0(z)) = u_0(z)$ to an arbitrary point.

If $(Y,g_Y)$ is a flat torus, this argument can be iteratively extended to all orders $\ell \geq 2$. Indeed, for all $k \in \{1,\ldots,\ell-1\}$, consider the component functions of the tensor $\nabla_y^k u(z,y)$ with respect to a parallel orthonormal frame of $Y$. Their average over $Y$ is always zero. Here we crucially use not only the flatness of $Y$ but also the fact that $\partial Y = \emptyset$. Note that Lemma \ref{lem:1} is false for $Y = [0,1] \subset \mathbb{R}$ and $\ell = 2$.

If $(Y,g_Y)$ is not a flat torus, we are not aware of an argument by direct estimates. Instead we can use the following indirect argument \cite[Lemma 3.3]{HeinTosatti2020}.

\begin{lemma}\label{lem:tensor-poincare}
    For all closed Riemannian manifolds $(Y,g_Y)$ and for all $p,q \in \mathbb{N}$ there exists a constant $C$ such that 
    for all type $(p,q)$ tensors $T$ on $Y$ we have that 
    \begin{equation*}
        \|\nabla_{g_Y} T\|_{C^0(Y,g_Y)} \leq C \|\nabla^2_{g_Y} T\|_{C^0(Y,g_Y)}.
    \end{equation*}
\end{lemma}

\begin{proof}
If this is false, then we can assume that there exists a sequence $T_i$ such that $\|\nabla^2 T_i\|_{C^0} \to 0$ as $i \to \infty$
but $\|\nabla T_i\|_{C^0} = 1$ for all $i$. Consider the space
$$\mathscr{P} = \{ T \in \Gamma(T^{p,q}Y) : \nabla T = 0 \}.$$
It is clear that $\dim\,\mathscr{P} < \infty$. Let $P_i$ be the $L^2$-orthonormal projection of $T_i$ onto $\mathscr{P}$ and let $Q_i = T_i - P_i$. Then $\|\nabla Q_i\|_{C^0(Y)}=1$, $\|\nabla^2 Q_i\|_{C^0(Y)}\rightarrow 0$, but $Q_i\perp_{L^2} \mathscr{P}$.

\medskip

\noindent {\bf Claim:} The sequence $Q_i$ is uniformly bounded in $C^0(Y)$.

\medskip

\noindent {\bf Proof of Claim:} Otherwise, after passing to a subsequence and normalizing, we get a sequence $Q_i'$ such that $\|Q'_i\|_{C^0} = 1$, $\|\nabla Q'_i\|_{C^0} \to 0$ and $\|\nabla^2Q'_i\|\to 0$. By Arzel\`a-Ascoli, a subsequence converges in $C^{1,\alpha}$ to some $Q'_\infty$ satisfying $\nabla Q'_\infty = 0$, hence $Q_\infty' \in \mathscr{P}$. But, also, $Q'_\infty \perp_{L^2} \mathscr{P}$, so $Q'_\infty = 0$, contradicting $\|Q'_\infty\|_{C^0} = 1$.

\medskip

Thanks to the above Claim we can again apply Arzel\`a-Ascoli to get $
Q_i \to Q_\infty$ in $C^{1,\alpha}$ after passing to a subsequence. For any $S \in \Gamma(T^{p,q+2}Y)$ consider the pairing  
\begin{equation*}
\langle \nabla^2 Q_i, S \rangle_{L^2}
=
\langle \nabla Q_i, \nabla^* S \rangle_{L^2} \to \langle \nabla Q_\infty, \nabla^*S \rangle_{L^2}.
\end{equation*}
Since the left-hand side goes to zero for any $S$, we see that  $\nabla Q_\infty\in C^{0,\alpha}$ is weakly parallel. Thus, $\nabla Q_\infty$ is smooth and parallel. Thus,
\begin{equation*}
\langle \nabla Q_\infty, \nabla Q_\infty \rangle_{L^2} = \langle \nabla^* \nabla Q_\infty, Q_\infty \rangle_{L^2} =\langle \operatorname{Tr}_{12}(\nabla^2 Q_\infty), Q_\infty \rangle_{L^2} = 0.
\end{equation*}
Thus, $\nabla Q_\infty=0$, contradicting $\|\nabla Q_\infty\|_{C^0}=1$. 
\end{proof}

Iterating Lemma~\ref{lem:tensor-poincare} on the tensors
$T = f,\nabla_{g_Y} f,\ldots, \nabla_{g_Y}^{\ell-2}f$, 
we get
\begin{equation}\label{eq:iterated-poincare} \|\nabla_{g_Y}f\|_{C^0(Y,g_Y)} \leq C_\ell \|\nabla_{g_Y}^\ell f\|_{C^0(Y,g_Y)},\end{equation}
where $C_\ell=C_\ell(Y,g_Y)$. Applying \eqref{eq:iterated-poincare} to $f = u_z = u(z,\cdot)$ for a fixed $z$, we get
$$\|\nabla_{t^2g_Y}u_z\|_{C^0} \leq C_\ell t^{\ell-1} \|(\nabla_{t^2g_Y})^\ell u_z\|_{C^0}\leq C_\ell t^{\ell-1}C.$$
Finally, by integrating $\nabla_{t^2 g_Y}u_z$ from $y_0(z)$ to an arbitrary point on $\{z\} \times Y$,
\begin{align*} &|u(z,y) - u_0(z)| \leq \mathrm{diam}(Y,t^2g_Y) \|\nabla_{t^2g_Y}u_z\|_{C^0} \leq \mathrm{diam}(Y,g_Y) C_\ell t^\ell C. \end{align*}
Taking the supremum over $B\times Y$ gives $\|u-u_0\|_{C^0(B\times Y)} \leq ACt^\ell$
for some constant $A=A(Y,g_Y,\ell)$ and we finish the proof of Lemma \ref{lem:1}.


\begin{acknowledgement}
YC was funded by the European Research Council (ERC) through the SiGMA project (Grant Agreement No. 101125012) and by the Erd\H{o}s Center postdoctoral fellowship at the Alfr\'ed R\'enyi Institute of Mathematics. HJH was funded by the Deutsche Forschungsgemeinschaft (DFG, German Research Foundation) under Germany’s Excellence Strategy EXC 2044/2–39068 5587 ``Mathematics M\"unster: Dynamics–Geometry–Structure'' and by the CRC 1442 ``Geometry: Deformations and Rigidity'' of the DFG.  All three authors would like to thank the Erd\H{o}s Center for hosting the Simons School on singular K\"ahlerian metrics and Hermitian geometry and for their hospitality, as well as all of the organizers involved.
\end{acknowledgement}

\eject

\end{document}